\documentclass{amsart}

\usepackage{graphicx}
\usepackage{amssymb}
\usepackage[latin1]{inputenc} 
\newtheorem{theorem}{Theorem}[section]
\usepackage{xcolor}

\newtheorem{lemma}[theorem]{Lemma}

\theoremstyle{definition}
\newtheorem{definition}[theorem]{Definition}

\theoremstyle{remark}
\newtheorem{remark}[theorem]{Remark}
\theoremstyle{proposition}
\newtheorem{proposition}{Proposition}
\theoremstyle{corollary}
\newtheorem{corollary}{Corollary}

\numberwithin{equation}{section}

\begin{document}

\title{$M\backslash L$ is not closed}


\author[D. Lima]{Davi Lima}

\address{Davi Lima: Instituto de Matem\'atica, UFAL, Av. Lourival Melo Mota s/n, Maceio, Alagoas, Brazil}

\email{davimat@impa.br}


\author[C. Matheus]{Carlos Matheus}
\address{Carlos Matheus: CMLS, \'Ecole Polytechnique, CNRS (UMR 7640),
91128, Palaiseau, France}
\email{matheus.cmss@gmail.com}

\author[C. G. Moreira]{Carlos Gustavo Moreira}
\address{Carlos Gustavo Moreira: School of Mathematical Sciences,   
Nankai University, Tianjin 300071, P. R. China, and 
IMPA, Estrada Dona Castorina 110, 22460-320, Rio de Janeiro, Brazil
}
\email{gugu@impa.br}

\author[S. Vieira]{Sandoel Vieira}
\address{Sandoel Vieira: IMPA, Estrada Dona Castorina, 110. Rio de Janeiro, Rio de Janeiro-Brazil.}
\email{sandoelpi@gmail.com}

\date{\today}


\dedicatory{To Christian Mauduit (in memoriam)}


\begin{abstract}
We show that $1+3/\sqrt{2}$ is a point of the Lagrange spectrum $L$ which is accumulated by a sequence of elements of  the complement $M\setminus L$ of the Lagrange spectrum in the Markov spectrum $M$. In particular, $M\setminus L$ is not a closed subset of $\mathbb{R}$, so that a question by T. Bousch has a negative answer.  
\end{abstract}

\maketitle

\vskip -.3 in
{\small \hskip 3in \it ``Aprendi novas palavras}
\vskip 0.001in
{\small \hskip 3in \it e tornei outras mais belas."}
\vskip .02in
{\small \hskip 2.8in \it (Carlos Drummond de Andrade)}

\section{Introduction}

The best constants of Diophantine approximations for irrational numbers and real indefinite quadratic binary forms are encoded by two closed subsets of the real line called the Lagrange and Markov spectra. The features of these spectra were intensively studied since the seminal works of A. Markov circa 1880, and we strongly recommend the book \cite{CF} of Cusick and Flahive and the survey article \cite{Bo} of Bombieri for beautiful reviews of some of the classical literature on this topic. 

A particularly challenging aspect about the structure of these spectra is the description of the nature of the set-theoretical difference $M\setminus L$ between the Lagrange spectrum $L$ and the Markov spectrum $M$. Indeed, Tornheim showed in 1955 that $L\subset M$, but the fact that $M\setminus L\neq \emptyset$ was established only in 1968 by Freiman \cite{Fr68}. 

In a series of three recent articles \cite{MM1}, \cite{MM2} and \cite{MM3}, the second and third authors of the present paper made some progress on the study of $M\setminus L$ by exhibiting three \emph{open} intervals $J_n$, $1\leq n\leq 3$, with the following properties:
\begin{itemize}
\item $J_1$, $J_2$ and $J_3$ are mutually disjoint intervals of sizes $\sim 2\cdot 10^{-10}$, $2\cdot 10^{-7}$, $10^{-10}$(resp.) near $3.11$, $3.29$ and $3.7$ (resp.); 
\item $J_n\cap L=\emptyset$ and $\partial J_n\subset L$ for all $1\leq n\leq 3$; 
\item $(M\setminus L)\cap J_n$, $1\leq n\leq 3$, are non-empty \emph{closed} subsets of $\mathbb{R}$ with positive Hausdorff dimensions.
\end{itemize}
The last item above led T. Bousch to ask whether $M\setminus L$ is a closed subset of $\mathbb{R}$. In a previous article \cite{LMMV}, we tried to solve negatively T. Bousch's question by giving strong evidence towards the \emph{possibility} that $3\in L\cap \overline{(M\setminus L)}$. Unfortunately, we could \emph{not} establish that $3\in \overline{M\setminus L}$ because we were unable\footnote{Partly due to the intricate combinatorial nature (explained in a clear way in Bombieri's survey article \cite{Bo}) of the bi-infinite sequences of $1$ and $2$ with Markov value $3$.} to prove a certain \emph{local uniqueness} property near $3$. 

In the present article, we show that $M\setminus L$ is not closed by establishing a local uniqueness property near $1+3/\sqrt{2}$ implying that $1+3/\sqrt{2}\in L\cap\overline{(M\setminus L)}$. 

The precise statement of our main result uses the intimate relationship between continued fractions and the Lagrange and Markov spectra. For this reason, let us now briefly recall some background material on continued fractions and Perron's characterisation of $L$ and $M$. 

\subsection{Basic features of continued fractions} We denote by 
$$\alpha=[a_0;a_1,a_2,\dots] = a_0 + \frac{1}{a_1+\frac{1}{a_2+\frac{1}{\ddots}}}$$
the continued fraction expansion of an irrational number $\alpha$. 

A basic comparison lemma asserts that if $\alpha=[a_0;a_1,\dots, a_n, a_{n+1},\dots]$ and $\tilde{\alpha}=[a_0;a_1,\dots, a_n, b_{n+1},\dots]$ with $a_{n+1}\neq b_{n+1}$, then  
\begin{equation}\label{ineq}
\alpha>\tilde{\alpha} \quad \mbox{if and only if} \quad (-1)^{n+1}(a_{n+1}-b_{n+1})>0. 
\end{equation}

The continued fraction expansion $\alpha=[a_0;a_1,\dots]$ of an irrational number $\alpha=\alpha_0$ can be recursively determined by setting $a_n=\lfloor\alpha_n\rfloor$ and $\alpha_{n+1} = \frac{1}{\alpha_n-a_n}$. Thus, we have that $\alpha_n=[a_n;a_{n+1}, ...]$. The convergents   
$$\frac{p_n}{q_n}:=[a_0;a_1,\dots,a_n]\in\mathbb{Q}$$ 
of $\alpha$ satisfy the recurrence relations $p_n=a_n p_{n-1}+p_{n-2}$, $q_n=a_n q_{n-1}+q_{n-2}$ and $p_{n+1}q_n-p_nq_{n+1}=(-1)^n$ (where $p_{-2}:=q_{-1}:=0$ and $p_{-1}:=q_{-2}:=1$). 

The quantity $\alpha_n$ is related to $\alpha=\alpha_0$ via a M\"obius transformation determined by the convergents $p_{n-1}/q_{n-1}$ and $p_{n-2}/q_{n-2}$: indeed, one has $\alpha=\frac{\alpha_n p_{n-1}+p_{n-2}}{\alpha_n q_{n-1}+q_{n-2}}$. Hence, if $\alpha=[a_0;a_1,\dots, a_n, a_{n+1},\dots]$ and $\tilde{\alpha}=[a_0;a_1,\dots,a_n,b_{n+1},\dots]$, then 
$$\alpha-\tilde{\alpha}=(-1)^n\frac{\tilde{\alpha}_{n+1}-\alpha_{n+1}}{q_n^2(\beta_n+\alpha_{n+1})(\beta_n+\tilde{\alpha}_{n+1})}$$ 
where $\beta_n:=\frac{q_{n-1}}{q_n}=[0;a_n,\dots,a_1]$. 

In general, a finite string $(a_1,\dots, a_l)\in(\mathbb{N}^*)^l$ determines a convergent  
$$[0;a_1,\dots,a_l] = \frac{p(a_1\dots a_l)}{q(a_1\dots a_l)}$$ 
verifying Euler's rule $q(a_1\dots a_l) = q(a_1\dots a_m) q(a_{m+1}\dots a_l) + q(a_1\dots a_{m-1}) q(a_{m+2}\dots a_l)$ for $1\leq m<l$. Consequently, $q(a_1\dots a_l) = q(a_l\dots a_1)$. In particular, if $(a_1,\dots, a_l)$ is a palindrome, then we also have $p(a_1\dots a_l) = p(a_l\dots a_1)$. 

\subsection{Markov and Lagrange spectra} The Markov value $m(\theta)$ of a bi-infinite sequence $\theta=(\theta_n)_{n\in\mathbb{Z}}\in(\mathbb{N}^*)^{\mathbb{Z}}$ is $m(\theta):=\sup\limits_{i\in\mathbb{Z}} \lambda_i(\theta)$, where 
$$\lambda_i(\theta):=[a_i;a_{i+1},a_{i+2},\dots]+[0;a_{i-1}, a_{i-2},\dots].$$

The \emph{Lagrange spectrum} $L$ is the closure of the set of Markov values of periodic words in $(\mathbb{N}^*)^{\mathbb{Z}}$ and the \emph{Markov spectrum} is the set $M:=\{m(\theta)<\infty: \theta\in(\mathbb{N}^*)^{\mathbb{Z}}\}$ of all possible finite Markov values. 

In this paper, we deal exclusively with Markov values below $\sqrt{12}$ and, for this reason, we can and do assume that all sequences appearing below belong to $\{1,2\}^{\mathbb{Z}}$. 

Moreover, we indicate the repetition of a character via subscripts: e.g., $1 2_3$ is the string $1222$. Furthermore, the periodic word obtained by infinite concatenation of the string $(a_1,\dots, a_l)$ is denoted $\overline{a_1,\dots, a_l}$. Finally, the zeroth position $a_0$ of a string $(a_{-m}, \dots, a_{-1}, a_0^*, a_1,\dots, a_n)$ is indicated by an asterisk (unless explicitly said otherwise).  

\subsection{Statement of the main result} For each $k\in\mathbb{N}$, consider the periodic word $\theta(\underline{\omega}_k)=\overline{\underline{\omega}_k}=...\underline{\omega}_k\underline{\omega}_k^{\ast}\underline{\omega}_k...$,  where the asterisk indicates the $0$-th position which occurs at the first $2$ in $\underline{\omega}_k$ from the left to the right,  associated to the finite string 
$$\underline{\omega}_k = 2_{2k-1},1,2_{2k},1,2_{2k+1},1$$ 
and define $\gamma_k^1\in\{1,2\}^{\mathbb{Z}}$, 
\begin{equation*}
\gamma_k^1:=\overline{2_{2k-1},1,2_{2k},1,2_{2k+1},1}2^*2_{2k-2},1,2_{2k},1,2_{2k+1},1,2_{2k-1},1,2_{2k},1,2_{2k-1},1,1,\overline{2},
\end{equation*}
where $*$ indicates the 0-position.

The main theorem of this article is:

\begin{theorem}\label{t.arabismos} The Markov values of $\theta(\underline{\omega}_k)$ and $\gamma_k^1$ satisfy:
\begin{itemize}
\item $m(\theta(\underline{\omega}_k)) < m(\gamma_k^1) < m(\theta(\underline{\omega}_{k-1}))$ for all $k\geq 3$; 
\item $\lim\limits_{k\to\infty}m(\theta(\underline{\omega}_k)) = 1+\frac{3}{\sqrt{2}}$. Thus, $1+\frac{3}{\sqrt{2}} \in L$. 
\item $m(\gamma_k^1)\in M\setminus L$ for all $k\geq 4$. 
\end{itemize}
In particular, $1+\frac{3}{\sqrt{2}}\in L\cap \overline{(M\setminus L)}$ and $M\setminus L$ is not a closed subset of $\mathbb{R}$. 
\end{theorem}

\begin{remark}\label{r.isolated-L} An interesting by-product of our arguments is the fact that $m(\theta(\underline{\omega}_k))$ is an isolated point of $L$ for all $k\geq4$: cf. Remark \ref{r.isolated-L-proof} below. 
\end{remark}

\subsection{Organisation of the article} 

The general strategy for the proof of Theorem \ref{t.arabismos} is similar to the arguments from our previous paper \cite{LMMV}: we want to construct a sequence of elements of $M\setminus L$ accumulating at $1+3/\sqrt{2}$ via a local uniqueness property and a replication mechanism. 

The main novelty of this article in comparison with \cite{LMMV} is the fact that we could establish Theorem \ref{t.local-uniqueness} below ensuring the local uniqueness property near $1+3/\sqrt{2}$. For this reason, we organise this paper as follows. 

After introducing in Section \ref{s.preliminaries} the crucial notions of prohibited and allowed strings, we discuss in Section \ref{L.U.} a list of prohibited and avoided permitting to prove the fundamental local uniqueness property in Theorem \ref{t.local-uniqueness} saying that a Markov value sufficiently close to $m(\gamma_k^1)$ must come from a sequence of the form $\dots 12_{2k+1}12^*2_{2k-2}1\dots$. Next, we implement in Sections \ref{s.extension} and \ref{s.replication} a replication mechanism (in the same spirit of Section 3 from our previous paper \cite{LMMV}) allowing to derive that $m(\gamma_k^1)\in M\setminus L$ for $k\geq 4$ because a Markov value close to $m(\gamma_k^1)$ must come from a sequence of the form $\overline{2_{2k-1}12_{2k}12_{2k+1}1}2^*2_{2k-2}12_{2k}12_{2k+1}12_{2k-1}12_{2k}12_4\dots$. Finally, we put together these ingredients to conclude the proof of Theorem \ref{t.arabismos} in Section \ref{s.end}. 

\section{Preliminaries}\label{s.preliminaries}

\subsection{Two important sequences converging to $1+3/\sqrt{2}$} 

\begin{lemma}\label{l.2-1}
For all $k\geq 2$, one has $\lambda_0(\theta(\underline{\omega}_k)) < \lambda_0(\gamma_k^1) < \lambda_0(\theta(\underline{\omega}_{k-1}))$. In particular, $(\lambda_0(\theta(\underline{\omega}_k)))_{k\geq 2}$ and $(\lambda_0(\gamma_k^1))_{k\geq 2}$ are decreasing sequences converging to $[2;\overline{2}]+[0;1,\overline{2}] = 1+3/\sqrt{2} = 3.12132034...$. 
\end{lemma}
\begin{proof}
The proof is a straightforward calculation using (\ref{ineq}).
\end{proof}

\subsection{Prohibited and avoided strings} Given a finite string $\underline{u}=(a_i)_{i=-m}^{n}$, let $$\lambda^{-}_i(\underline{u}):=\min \{[a_i;a_{i+1},...,a_n,\theta_1]+[0;a_{i-1},...,a_{-m},\theta_2]: \theta_1,\theta_2\in \{1,2\}^{\mathbb{N}}\},$$
 and
 $$\lambda^{+}_i(\underline{u}):=\max \{[a_i;a_{i+1},...,a_n,\theta_1]+[0;a_{i-1},...,a_{-m},\theta_2];\theta_1,\theta_2\in \{1,2\}^{\mathbb{N}}\}.$$
 
 \begin{definition}
 	We say that $\underline{u}=(a_i)_{i=-m}^{n}$ is:
	\begin{itemize} 
	\item $k$-\emph{prohibited} whenever $\lambda^{-}_i(\underline{u})>\lambda_0(\gamma^1_k)$, for some $-m\leq i\leq n$.     	\item $k$-\emph{avoided} if $\lambda^+_0(\underline{u})<\lambda_0(\theta(\underline{\omega}_k))$.
	\end{itemize}
	A word $\theta\in\{1,2\}^{\mathbb{Z}}$ is $(k,\lambda)$-\emph{admissible} when $\lambda_0(\theta(\underline{\omega}_k))<m(\theta)=\lambda_0(\theta)<\lambda$.
 \end{definition}

These notions are the key to obtain local uniqueness and self-replication properties: in a nutshell, the local uniqueness is based on the construction of a finite set of prohibited and avoided strings and the self-replication relies on a finite set of prohibited strings. In this setting, our main goal is to setup local uniqueness and self-replication properties in such a way that the Markov value of any $(k,\lambda_k)$-admissible word belongs to $M\setminus L$ whenever $\lambda_k$ is close to $m_k = m(\gamma_k^1)$.

\section{Local uniqueness}\label{L.U.}

We begin this section by the following lemma:
\begin{lemma} \label{p1}
\begin{itemize}
\item[i)]$\lambda^-_0(12^*1)>3.154$
\item[ii)] $\lambda^+_0(22^*2)<\lambda^+_0(112^*2)<3.057$
\end{itemize}
\end{lemma}

In particular, up to transposition, if $\theta$ is $(k,3.154)$-admissible, then $\theta=...2212^*2...$.
 
 On the other hand, if $\theta=...2_a12^*2_b...$ with $a>2k+1$ and $b>2k-2$, then $\lambda^+_0(\theta)<\lambda_0(\theta(\underline{\omega}_k)),$ because
 $$[2;2_{b-1},2,...]<[2;2_{2k-2},1,...] \quad \mbox{and} \quad [0;1,2_{a-1},2,...]<[0;1,2_{2k+1},1,...].$$
 
Thus, a $(k,3.154)$-admissible word $\theta$ falls into one of the following categories:
 
\begin{enumerate}
 
 \item[$A_{a,b}$:] $\theta=...12_a12^*2_b1...$ with $a\le 2k+1$ and $b\le 2k-2$,
 
 \item[$B_a$:] $\theta=...12_a12^*2_{2k-1}...$, with $a\leq 2k+1$.
 
 \item[$C_b$:] $\theta=...2_{2k+2}12^*2_b1...$ with $b\leq 2k-2$.
 
 \end{enumerate}
 
 The main theorem of this section is:
 
 \begin{theorem}\label{t.local-uniqueness} For each $k\geq 3$, there is a constant $\lambda_k^{(1)}>\lambda_0(\gamma_k^1)$ such that any $(k,\lambda_k^{(1)})$-admissible word $\theta$ falls into the category $A_{2k+1, 2k-2}$, i.e., has the form 
 $$\theta =...12_{2k+1}12^*2_{2k-2}1...$$ 
 \end{theorem}
 
 The proof of this result consists into excluding all other categories $B_a$, $C_b$ and $A_{a,b}$ and it occupies the remainder of this section. 
 
\subsection{Ruling out $B_a$ with $a$ even} 

\begin{lemma}\label{L.U1} If $u=12_{2j}12^*2_{2k-1}, 0\leq j\le k$, then $\lambda^+_0(u)<m(\theta(\underline{\omega}_k))$. 
\end{lemma}

\begin{proof}
Note that $$[2;2_{2k-2},2,...]<[2;2_{2k-2},1] \quad \mbox{and} \quad [0;1,2_{2j},1,...]<[0;1,2_{2j},2_{2k-2j},2,...]$$
\end{proof}
 
 \subsection{Ruling out $B_a$ with $a$ odd} 
 
 \begin{lemma}\label{L.U2}
Let $u_j=12_{2j+1}12^*2_{2k-1}$ with $0\leq j\leq k$. Then, 
$$\lambda^+_0(u_{k})<\lambda^+_0(u_{k-1})<\lambda_0(\theta(\underline{\omega}_k))\quad \mbox{and} \quad \lambda_0(\gamma^1_k)<\lambda^-_0(u_{k-2})\leq\lambda^-_0(u_j) \,\forall \, j\leq k-2.$$
 \end{lemma}

 \begin{proof}
 Write $\lambda^+_0(u_{k-1})=[2;2_{2k-1},\overline{2,1}]+[0;1,2_{2k-1},1,\overline{2,1}]:=A+B$ and 
	 $$\lambda_0(\theta(\underline{\omega}_k))>[2;2_{2k-2},1,\overline{1,2}]+[0;1,2_{2k+1},1,\overline{1,2}]:=C+D.$$
  Note that $C-A=[0;2_{2k-2},1,\overline{1,2}]-[0;2_{2k-1},\overline{2,1}]$, so that 
	 $$C-A=\dfrac{[2;\overline{2,1}]-[1;\overline{1,2}]}{q^2(2_{2k-2})([2;\overline{2,1}]+\beta(2_{2k-2}))([1;\overline{1,2}]+\beta(2_{2k-2}))}.$$

  Moreover, $D-B=[0;1,2_{2k+1},1,\overline{1,2}]-[0;1,2_{2k-1},1,\overline{2,1}]$, so that 
  	$$B-D=\dfrac{[2;2,1,\overline{1,2}]-[1;\overline{2,1}]}{q^2(12_{2k-1})([2;2,1,\overline{1,2}]+\beta(12_{2k-1}))([1;\overline{2,1}]+\beta(12_{2k-1}))}.$$
 This implies that
 	$$\dfrac{C-A}{B-D}=\dfrac{q^2(12_{2k-1})}{q^2(2_{2k-2})}\cdot X \cdot Y,$$
where 
$$X=\dfrac{[2;\overline{2,1}]-[1;\overline{1,2}]}{[2;2,1,\overline{1,2}]-[1;\overline{2,1}]}>0.62$$	
and	
	$$Y=\dfrac{([2;2,1,\overline{1,2}]+\beta(12_{2k-1}))([1;\overline{2,1}]+\beta(12_{2k-1}))}{[2;\overline{2,1}]+\beta(2_{2k-2}))([1;\overline{1,2}]+\beta(2_{2k-2}))}>0.62.$$ 
Since $q(12_{j})=q(2_j1)=q(2_j)+q(2_{j-1})$, we have 
	$$\dfrac{C-A}{B-D}=\left(\dfrac{q(2_{2k-1})}{q(2_{2k-2})}+1 \right)^2\cdot X \cdot Y=(3+\beta(2_{2k-2}))^2\cdot X \cdot Y>1.$$
In particular, $C-A>B-D$ and  
$$\lambda^+_0(u_{k-1}) < \lambda_0(\theta(\underline{\omega}_k)).$$

Next, we write
	$$\lambda^-_0(u_{k-2})=[2;2_{2k-1},\overline{1,2}]+[0;1,2_{2k-3},1,\overline{1,2}]:=A'+B'$$
and
	$$\lambda_0(\gamma^1_k)<[2;2_{2k-2},1,\overline{2,1}]+[0;1,2_{2k+1},1,\overline{2,1}]:=C'+D'.$$	
Note that
	$$C'-A'=\dfrac{[2;\overline{1,2}]-[1;\overline{2,1}]}{q^2(2_{2k-2})([2;\overline{1,2}]+\beta(2_{2k-2}))([1;\overline{2,1}]+\beta(2_{2k-2}))}$$ 
	and 
	$$B'-D'=\dfrac{[2;2,2,\overline{2,1}]-[1;\overline{1,2}]}{q^2(12_{2k-3})([2;2,2,\overline{2,1}]+\beta(12_{2k-3}))([1;\overline{1,2}]+\beta(12_{2k-3}))}.$$
	Therefore,
	$$\dfrac{B'-D'}{C'-A'}=\dfrac{q^2(2_{2k-2})}{q^2(12_{2k-3})}\cdot X' \cdot Y'=\left(1+\dfrac{1}{1+\beta(2_{2k-3})}\right)^2 \cdot X' \cdot Y',$$
where
	$$X'=\dfrac{[2;2,2,\overline{2,1}]-[1;\overline{1,2}]}{[2;\overline{1,2}]-[1;\overline{2,1}]}>0.4983$$	
and	
	$$Y'=\dfrac{([2;\overline{1,2}]+\beta(2_{2k-2}))([1;\overline{2,1}]+\beta(2_{2k-2}))}{([2;2,2,\overline{2,1}]+\beta(12_{2k-3}))([1;\overline{1,2}]+\beta(12_{2k-3}))}>0.91.$$	
Since $\left(1+\dfrac{1}{1+\beta(2_{2k-3})}\right)^2>2.9$ (because $\beta(2_{2k-3}) \leq [0;2,2,2]$ for $k\ge 3$), we get 
	$$\dfrac{B'-D'}{C'-A'}>2.9\cdot 0.49 \cdot 0.91>1.$$
In particular, $\lambda^-_0(u_{k-2})> \lambda_0(\gamma^1_k)$. This completes the proof of the lemma. 		
 \end{proof}

 \subsection{Ruling out $C_b$ with $b$ odd} 
 
 \begin{lemma}\label{L.U3}
If $u=2_{2k+2}12^*2_{2m-1}1$ with $m<k$, then $\lambda_0^+(u)<\lambda_0(\theta(\underline{\omega}_k)).$ 
 
 \end{lemma}
\begin{proof}
Note that $[2;2_{2m-1},1,...]<[2;2_{2m-1},2_{2k-2m-1},...]$ and $[0;1,2_{2k+1},2,...]<[0;1,2_{2k+1},1,...]$. 
\end{proof}

 \subsection{Ruling out $C_b$ with $b$ even}  
 
 \begin{lemma}\label{cL.U1}
 Let $A=[a_0;\underline{a},\alpha]$, $B=[b_0;\underline{b},\zeta]$, $C=[a_0;\underline{a},\gamma]$ and $D=[b_0;\underline{b},\eta]$ with $\underline{a}$, resp. $\underline{b}$, a finite string of $1$ and $2$ of length $\geq 2$, resp. $\geq 3$ and $\alpha, \zeta, \gamma, \eta\in \{1,2\}^{\mathbb{N}}$, $\alpha_1\neq\gamma_1$, $\zeta_1\neq \eta_1$. Suppose that $q(\underline{b})\geq 3q(\underline{a})$. Then, 
 	$$A+B>C+D \quad \mbox{if} \quad A>C \ \mbox{and} \ D>B$$
	and 
	$$C+D>A+B \quad \mbox{if} \quad C>A \ \mbox{and} \ B>D.$$	
Moreover, the same statement is also true when the assumptions $\underline{a}$ has length $\geq 2$ and/or $\underline{b}$ has length $\geq 3$ are replaced by $\underline{a}$ starts with $2$ and/or $\underline{b}$ starts with $1$. 
 \end{lemma}
 
 \begin{proof}
 If $A>C$ and $D>B$, we have
 	$$A-C=\dfrac{|[\gamma]-[\alpha]|}{q^2(\underline{a})([\alpha]+\beta(\underline{a}))([\gamma]+\beta(\underline{a}))}$$
and
	$$D-B=\dfrac{|[\zeta]-[\eta]|}{q^2(\underline{b})([\zeta]+\beta(\underline{b}))([\eta]+\beta(\underline{b}))}.$$
	Consider
	$$X=\dfrac{|[\gamma]-[\alpha]|}{|[\zeta]-[\eta]|}$$	
	and
	$$Y=\dfrac{([\zeta]+\beta(\underline{b}))([\eta]+\beta(\underline{b}))}{([\alpha]+\beta(\underline{a}))([\gamma]+\beta(\underline{a}))}.$$
Therefore, 
	$$\dfrac{A-C}{D-B}=\dfrac{q^2(\underline{b})}{q^2(\underline{a})}\cdot X\cdot Y.$$
Since $\underline{a}$ and $\underline{b}$ are finite strings of $1$ and $2$ with lengths $\geq 2$ and $\geq 3$ (resp.) and $\alpha, \zeta, \gamma, \eta\in \{1,2\}^{\mathbb{N}}$ with $\alpha_1\neq\gamma_1$, $\zeta_1\neq\eta_1$, we have that $X\geq \frac{1+[0;\overline{2,1}]-[0;\overline{1,2}]}{1+[0;\overline{1,2}]-[0;\overline{2,1}]}$, $Y\geq \frac{(1+[0;\overline{2,1}]+[0;2,1,2,1])^2}{(2+[0;\overline{1,2}]+[0;1,2,1])^2}$ and $X\cdot Y > \dfrac{1}{9}$. On the other hand, we are assuming that  $\dfrac{q^2(\underline{b})}{q^2(\underline{a})}\geq 9$. Thus, 
	$$\dfrac{A-C}{D-B}>1.$$	
The other cases are analogous.	
 \end{proof}
 
 \begin{lemma}\label{L.U4}
 Let $u_m=2_{2k+2}12^*2_{2m}1$. If $m\leq k-2$ and $k\geq 3$, then $\lambda^-_0(u_m)\ge \lambda^-_0(u_{k-2})>\lambda_0(\gamma^1_k).$
 \end{lemma}

\begin{proof}
Write $\lambda^-_0(u_{k-2})=[2;2_{2k-4},1,\overline{1,2}]+[0;1,2_{2k+2},\overline{1,2}]:=A+B$ and 
	$$\lambda_0(\gamma^1_k)<[2;2_{2k-2},1,\overline{2,1}]+[0;1,2_{2k+1},\overline{1,2}]:=C+D.$$
If we take $\underline{a}=2_{2k-4}$ and $\underline{b}=12_{2k+1}$ we have by Euler's rule $q(12_{2k+1})>4q(2_{2k-4})$. Since $A>C$ and $D>B$, we deduce from Lemma \ref{cL.U1} that $A+B>C+D$. 
\end{proof}

\begin{lemma} \label{bi}
Let $\alpha$ be a finite string. We have:
\begin{itemize} 
\item[i)] $q(\alpha2)/3<q(\alpha)<q(\alpha2)/2$ and $4q(\alpha2)/3<q(\alpha21)<3q(\alpha2)/2$
\item[ii)]$7q(\alpha2_4)/17<q(\alpha2_3)<5q(\alpha2_4)/12$ and $24q(\alpha2_4)/17<q(\alpha2_41)<17q(\alpha2_4)/12$
\end{itemize}
\end{lemma}

 \begin{lemma} \label{l.Ck-1}
 Let $\theta=2_{2k+2}12^*2_{2k-2}1$ with $k\geq 3$. Then, $\lambda^+_0(\theta1)<\lambda^+_0(\theta22) <\lambda_0(\theta(\underline{\omega}_k)).$
 \end{lemma}
 \begin{proof}
 Note that $\lambda^+_0(\theta1)<\lambda^+_0(\theta22)$ because $[0;2_{2k-2},1,1,...]<[0;2_{2k-2},1,2,...]$. In order to prove that
 $\lambda^+_0(\theta22) <\lambda_0(\theta(\underline{\omega}_k))$, let us write
 	$$\lambda^+_0(\theta22)=[2;2_{2k-2},1,2_2,\overline{2,1}]+[0;1,2_{2k+2},\overline{2,1}]:=C+D$$
and
	$$\lambda_0(\theta(\underline{\omega}_k))>[2;2_{2k-2},1,2_{5},\overline{2,1}]+[0;1,2_{2k+1},1,2_5,\overline{2,1}]:=A+B.$$	
Observe that 	
$$B-D=\dfrac{[2;\overline{2,1}]-[1;2_5,\overline{2,1}]}{q^2_{2k+2}([2;\overline{2,1}]+\beta)([1;2_5,\overline{2,1}]+\beta)}$$
and
$$C-A=\dfrac{[2;2,\overline{2,1}]-[2;2_4,\overline{2,1}]}{\tilde{q}^2_{2k-1}([2;2_{2},\overline{2,1}]+\tilde{\beta})([2;2_5,\overline{2,1}]+\tilde{\beta})},$$
where $q_{2k+2}=q(12_{2k+1})$, $\tilde{q}_{2k-1}=q(2_{2k-2}1)$, $\beta=[0;2_{2k+1},1]$ and $\tilde{\beta}=[0;1,2_{2k-2}]$.

Thus, 
$$\dfrac{B-D}{C-A}=X\cdot Y\cdot \dfrac{\tilde{q}^2_{2k-1}}{q^2_{2k+2}},$$
where 
	$$X=\dfrac{[2;\overline{2,1}]-[1;2_5,\overline{2,1}]}{[2;2,\overline{2,1}]-[2;2_4,\overline{2,1}]}> 112.25,$$
	and, since $\beta<[0;\overline{2}]$ and $\tilde{\beta}>[0;1,2_3]$, 
	$$Y=\dfrac{[2;2,\overline{2,1}]+\tilde{\beta})([2;2_4\overline{2,1}]+\tilde{\beta})}{([2;\overline{2,1}]+\beta)([1;2_5,\overline{2,1}]+\beta)}>1.9201.$$
	By Lemma \ref{bi} ii), we have $q_{2k+2}=12q(12_{2k-2})+5q(12_{2k-3})<q(12_{2k-2})(12+5\cdot \frac{5}{12})$. Since $q(12_{2k-2})=\tilde{q}_{2k-1}$, we get $\dfrac{\tilde{q}^2_{2k-1}}{q^2_{2k+2}}>\left(\dfrac{12}{169}\right)^2.$
	Therefore,
	$$\dfrac{B-D}{C-A}=112.25\cdot 1.92 \cdot \left(\dfrac{12}{169}\right)^2 >1.08>1.$$
\end{proof}
 
\subsection{Ruling out $A_{a,b}$ with $a$ odd and $b$ even}  We want to show that this case essentially never occurs, except when $a=2k+1$ and $b=2k-2$. In order to see this fact, we analyse now the following cases:

\begin{itemize}
\item[I)] $a<2k+1$ odd and $b<2k-2$ even;
\item[II)] $a=2k+1$ and $b<2k-2$ even;
\item[III)] $a<2k+1$ odd and $b=2k-2$;
\item[IV)] $a=2k+1$ and $b=2k-2$.
\end{itemize} 

The next lemma ensures that the case $I)$ essentially never occurs:
\begin{lemma}\label{L.U5}
If $u=12_{2j+1}12^*2_{2m}1$ with $m<k-1$, $j<k$, then $\lambda^-_0(u)>\lambda_0(\gamma^1_k).$
\end{lemma}
\begin{proof}
Note that $[2;2_{2m},1,...]>[2;2_{2k-2},1,...]$ and $[0;1,2_{2j+1},1,...]>[0;1,2_{2k+1},1,...]$ whenever $m<k-1$ and $j<k$. 
\end{proof}

The next lemma guarantees that the case $II)$ essentially never occurs: 
\begin{lemma} \label{bpodd}
If $2m\leq 2k-4$, then $\lambda^-_0(2_{2k-2}12^*2_{2m}1)\ge \lambda^-_0(2_{2k-2}12^*2_{2k-4}1) >\lambda_0(\gamma_k^1).$
 \end{lemma}
 \begin{proof}
 Let us write $\lambda^-_0(2_{2k-2}12^*2_{2k-4}1)=[2;2_{2k-4},1,\overline{1,2}]+[0;1,2_{2k-2},1,\overline{1,2}]:=A+B$ and   $\lambda_0(\gamma^1_k)<[2;2_{2k-2},1,\overline{2,1}]+[0;1,2_{2k+1},1,\overline{2,1}]:=C+D.$ In particular, $A>C$ and $D>B$. Take $\underline{a}=2_{2k-4}$ and $\underline{b}=12_{2k-1}$. By Euler's rule $q(12_{2k-1})>4q(2_{2k-4})$. By Lemma \ref{cL.U1} we have
 	$$A+B>C+D.$$
	This completes the argument because $[0;2_{2k-4},1,...]\le [0;2_{2m},1,...]$ and, \emph{a fortiori}, $\lambda^-_0(2_{2k-2}12^*2_{2m}1)\ge \lambda^-_0(2_{2k-2}12^*2_{2k-4}1)$ whenever $2m\le 2k-4$.
 \end{proof}
The case $III)$ essentially never occurs thanks to Lemma \ref{p1} i) and the next two lemmas:
\begin{lemma} \label{l.Aoddeveniii}
If $2j+1\leq 2k-3$ and $k\geq 3$, then $\lambda^-_0(12_{2j+1}12^*2_{2k-2})>\lambda^-_0(12_{2k-3}12^*2_{2k-2}) >\lambda_0(\gamma_k^1).$
 \end{lemma}
 \begin{proof}
 We begin by noticing that $q(12_{2k-3})=q(2_{2k-3})+q(2_{2k-4})$ and $q(2_{2k-2})=2q(2_{2k-3})+q(2_{2k-4}).$
Therefore,
	$$\dfrac{q(2_{2k-2})}{q(12_{2k-3})}=1+\dfrac{1}{1+ \beta(2_{2k-3})}> 1.6.$$	
Next, we write $\lambda^-_0(12_{2k-3}12^*2_{2k-2})=[2;2_{2k-2},\overline{2,1}]+[0;1,2_{2k-3},1,\overline{1,2}]:=A+B$ and  $\lambda_0(\gamma^1_k)<[2;2_{2k-2},1,\overline{2,1}]+[0;1,2_{2k+1},1,\overline{2,1}]:=C+D.$ It follows that
	$$C-A=\dfrac{[2;\overline{1,2}]-[1;\overline{2,1}]}{q^2(2_{2k-2})([2;\overline{1,2}]+\beta(2_{2k-2})([1;\overline{2,1}]+\beta(2_{2k-2}))}$$
and 
	$$B-D=\dfrac{[2;2,2,2,1,\overline{2,1}]-[1;\overline{1,2}]}{q^2(12_{2k-3})([2;2,2,2,1,\overline{2,1}]+\beta(12_{2k-3}))([1;\overline{1,2}]+\beta(12_{2k-3}))}.$$
	Therefore, $\dfrac{B-D}{C-A}=\dfrac{q^2(2_{2k-2})}{q^2(12_{2k-3})}\cdot X\cdot Y,$ where
	$$X=\dfrac{[2;2,2,2,1,\overline{2,1}]-[1;\overline{1,2}]}{[2;\overline{1,2}]-[1;\overline{2,1}]}>0.498$$
and
	$$Y=\dfrac{([2;\overline{1,2}]+\beta(2_{2k-2})([1;\overline{2,1}]+\beta(2_{2k-2}))}{([2;2,2,2,1,\overline{2,1}]+\beta(12_{2k-3}))([1;\overline{1,2}]+\beta(12_{2k-3}))}.$$		
Note that
	$$Y>\dfrac{([2;\overline{1,2}]+0.4)([1;\overline{2,1}]+0.4)}{([2;2,2,2,1,\overline{2,1}]+0.5)([1;\overline{1,2}]+0.5)}>0.85.$$	
	Thus
	$$\dfrac{B-D}{C-A}>2.56\cdot 0.498 \cdot 0.85>1.$$
 \end{proof}
 
 \begin{lemma} \label{ooe}
  Let $\theta=12_{2k-1}12^*2_{2k-2}1$ with $k\geq 3$. We have:
 \begin{itemize}
 \item[i)]$\lambda^-_0(2\theta22)>\lambda^-_0(1\theta22) >\lambda_0(\gamma_k^1);$
 \item[ii)]$\lambda^+_0(1\theta1)<\lambda^+_0(22\theta1) <\lambda_0(\theta(\underline{\omega}_k)).$
 \end{itemize}
 \end{lemma}
 \begin{proof} Let us first establish i). For this sake, we write $\lambda_0^-(1\theta22) = [2; 2_{2k-2}122\overline{12}]+[0;12_{2k-1}11\overline{21}] :=A+B$ and $\lambda_0(\gamma_k^1)<[2;2_{2k-2}12_5\overline{12}]+[0;12_{2k+1}1\overline{21}]:=C+D$. Note that 
$$C-A=\dfrac{[2;2,2,\overline{1,2}]-[1;\overline{2,1}]}{q^2(2_{2k-2}122)([2;2,2,\overline{1,2}]+\beta(2_{2k-2}122)([1;\overline{2,1}]+\beta(2_{2k-2}122))}$$
and 
	$$B-D=\dfrac{[2;2,1,\overline{2,1}]-[1;\overline{1,2}]}{q^2(12_{2k-1})([2;2,1,\overline{2,1}]+\beta(12_{2k-1}))([1;\overline{1,2}]+\beta(12_{2k-1}))}.$$ 
Hence, 
$\dfrac{B-D}{C-A}=\dfrac{q^2(2_{2k-2}122)}{q^2(12_{2k-1})}\cdot X\cdot Y,$ where
	$$X=\dfrac{[2;\overline{2,1}]-[1;\overline{1,2}]}{[2;2,\overline{2,1}]-[1;\overline{2,1}]}=0.6$$
and
	$$Y=\dfrac{([2;2,2,\overline{1,2}]+\beta(2_{2k-2}122)([1;\overline{2,1}]+\beta(2_{2k-2}122))}{([2;2,1,\overline{2,1}]+\beta(12_{2k-1}))([1;\overline{1,2}]+\beta(12_{2k-1}))}.$$ 
Since $[0;2,2,2,1]<\beta(2_{2k-2})<[0;2,2,2]$, $\beta(12_{2k-1})<[0;2,2,2]$ and $\beta(2_{2k-2}122)>[0;2,2,1,2]$, we have $$\dfrac{q(2_{2k-2}122)}{q(12_{2k-1})} = \dfrac{7+\beta(2_{2k-2})}{3+\beta(2_{2k-2})}>2.1692,$$ 
$Y>0.84993$ and, \emph{a fortiori}, 
$$\frac{B-D}{C-A}>2.399>1.$$ 

Let us now prove ii). In this direction, we write $\lambda^+_0(22\theta1) = [2;2_{2k-2},1,1,\overline{1,2}] + [0;1,2_{2k-1},1,2,2, \overline{2,1}] := A'+B'$ and $\lambda_0(\theta(\underline{\omega}_k))>[2;2_{2k-2},1,2,2, \overline{1,2}]+[0;1,2_{2k+1},1, \overline{1,2}]:=C'+D'$. Observe that 
$$\dfrac{C'-A'}{B'-D'}=\dfrac{q^2(12_{2k-1})}{q^2(2_{2k-2}1)}\cdot X'\cdot Y',$$ 
where 
$$X'=\dfrac{[2;\overline{2,1}]-[1;\overline{1,2}]}{[2;2,1,\overline{1,2}]-[1;2,2,\overline{2,1}]}>0.65$$ 
and 
$$Y'=\dfrac{([2;2,1,\overline{1,2}]+\beta(12_{2k-1})([1;2,2,\overline{2,1}]+\beta(12_{2k-1}))}{([2;\overline{2,1}]+\beta(2_{2k-2}1))([1;\overline{1,2}]+\beta(2_{2k-2}1))}.$$ 
Since $\beta(2_{2k-2}1)<[0;122]$ and $[0;2222]<\beta(12_{2k-1})<\beta(12_{2k-2})<[0;222]$, we see that $Y'>0.67$, 
$$\dfrac{q(12_{2k-1})}{q(2_{2k-2}1)} = 2+\beta(12_{2k-2})>2.41$$ 
and, \emph{a fortiori}, $(C'-A')/(B'-D')>2.529>1$. 
 \end{proof} 
 
\subsection{Ruling out $A_{a,b}$ with $a$ even and $b$ odd}  This case essentially never occurs.

\begin{lemma}\label{L.U6}
If $u=12_{2j}12^*2_{2m+1}1$ with $2j\leq 2k+1$ and $2m+1\leq 2k-2$, then $\lambda^+_0(u)<\lambda_0(\theta(\underline{\omega}_k)).$
\end{lemma}

\begin{proof}
Note that $[2;2_{2m+1},1,...]<[2;2_{2k-2},1,...]$ and $[0;1,2_{2j},1,...]<[0;1,2_{2k+1},1,...]$ whenever $2m+1\leq 2k-2$ and $2j\leq 2k+1$.
\end{proof}

\subsection{Ruling out $A_{a,b}$ with $a,b$ even}  This case essentially never occurs.

\begin{lemma}\label{L.U7}
Let $u_{j,m}=12_{2j}12^*2_{2m}1$ with $j\leq k$ and $m\leq k-1$. We have:
\begin{itemize}
\item[i)] If $k-1\geq m>j$, then  $\lambda^+_0(u_{j,m})<\lambda_0(\theta(\underline{\omega}_k))$;
\item[ii)] If $k-1> m$ and $j>m$, then  $\lambda^-_0(u_{j,m})>\lambda_0(\gamma_k^1)$;
\item[iii)] If $k-1> m=j$, then $\lambda^-_0(u_{j,m}22)>\lambda_0(\gamma_k^1)$ and $\lambda^-_0(1u_{j,m}1)>\lambda^-_0(22u_{j,m}1)>\lambda_0(\gamma_k^1)$;
\item[iv)] If $j=m=k-1$, then $\lambda^+_0(u_{k-1,k-1})<\lambda_0(\theta(\underline{\omega}_k))$;
\item[v)] If $m=k-1$ and $j=k$, then $\lambda^+_0(u_{k,k-1}1)<\lambda^+_0(u_{k,k-1}22)<\lambda_0(\theta(\underline{\omega}_k))$.
\end{itemize}
\end{lemma}

\begin{proof} Let us prove i). For this sake, write $\lambda_0^+(u_{j,m})=[2;2_{2m}1\overline{21}]+[0;12_{2j}1\overline{12}]:=B+A$ and $\lambda_0(\theta(\underline{\omega}_k))>[2;2_{2k-2}1\overline{12}]+[0;12_{2k+1}1\overline{12}]:=D+C$. By Lemma \ref{cL.U1}, we get $A+B<C+D$ because $C>A$, $B>D$ and $\dfrac{q(2_{2k-2}1)}{q(12_{2j})}\geq \dfrac{q(2_{2m})}{q(12_{2j})}\geq \dfrac{q(2_{2j+2})}{q(12_{2j})}=\dfrac{5+2\beta(2_{2j})}{1+\beta(2_{2j})}\geq \dfrac{5+2[0;22]}{1+[0;2]}>3$. 

Let us now establish ii). In this direction, we set $\lambda^-_0(u_{j,m})=[2;2_{2m}1\overline{12}]+[0;12_{2j}1\overline{21}]:=A'+B'$ and $\lambda_0(\gamma_k^1)<[2;2_{2k-2}1\overline{21}]+[0;12_{2k+1}1\overline{21}]=C'+D'$. Since $A'>C'$, $B'<D'$ and $\dfrac{q(12_{2j})}{q(2_{2m})} = \dfrac{q(2_{2j})+q(2_{2j-1})}{q(2_{2m})}\geq \dfrac{q(2_{2m+2})+q(2_{2m+1})}{q(2_{2m})} > 3$, it follows from Lemma \ref{cL.U1} that $A'+B'>C'+D'$. 

Let us show iii). For this purpose, we denote $\lambda^-_0(u_{j,m}22) = [2;2_{2m}122\overline{12}]+[0;12_{2m}1\overline{21}]:=A''+B''$, $\lambda^-_0(22u_{j,m}1)=[2;2_{2m}11\overline{21}]+[0;12_{2m}122\overline{21}]:=A'''+B'''$ and $\lambda_0(\gamma_k^1)<[2;2_{2k-2}1\overline{21}]+[0;12_{2k+1}1\overline{21}]=C'+D'$. Observe that 
$$\dfrac{A''-C'}{D'-B''} = \dfrac{q^2(12_{2m})}{q^2(2_{2m})} \cdot X''\cdot Y'' \quad \textrm{and} \quad \dfrac{A'''-C'}{D'-B'''} = \dfrac{q^2(12_{2m})}{q^2(2_{2m})} \cdot X'''\cdot Y'''$$ 
where 
$$X''=\dfrac{[2;2_{2k-2m-3}1\overline{21}]-[1;22\overline{12}]}{[2;2_{2k-2m}1\overline{21}]-[1;\overline{21}]}, \quad X'''=\dfrac{[2;2_{2k-2m-3}1\overline{21}]-[1;1\overline{12}]}{[2;2_{2k-2m}1\overline{21}]-[1;22\overline{21}]},$$
$$Y''=\dfrac{([2;2_{2k-2m}1\overline{21}]+\beta(12_{2m}))([1;\overline{21}]+\beta(12_{2m}))}{([2;2_{2k-2m-3}1\overline{21}]+\beta(2_{2m}))([1;22\overline{12}]+\beta(2_{2m}))}$$
and 
$$Y'''=\dfrac{([2;2_{2k-2m}1\overline{21}]+\beta(12_{2m}))([1;22\overline{21}]+\beta(12_{2m}))}{([2;2_{2k-2m-3}1\overline{21}]+\beta(2_{2m}))([1;1\overline{12}]+\beta(2_{2m}))}.$$
Since $\dfrac{q(12_{2m})}{q(2_{2m})} = 1+\beta(2_{2m})\geq 1+[0;22]=1.4$, 
$$X''\geq \dfrac{[2;21\overline{21}]-[1;22\overline{12}]}{[2;2_41\overline{21}]-[1;\overline{21}]}>0.899, \quad 
X'''\geq \dfrac{[2;21\overline{21}]-[1;1\overline{12}]}{[2;2_41\overline{21}]-[1;22\overline{21}]}>0.787,$$ 
$$Y''\geq \dfrac{([2;2_31\overline{21}]+[0;22])([1;\overline{21}]+[0;22])}{([2;2\overline{21}]+[0;2])([1;22\overline{12}]+[0;2])}>0.884,$$
and 
$$Y'''\geq \dfrac{([2;2_31\overline{21}]+[0;22])([1;22\overline{21}]+[0;22])}{([2;2\overline{21}]+[0;2])([1;1\overline{12}]+[0;2])} > 0.839,$$ 
we see that $\dfrac{A''-C'}{D'-B''}>1.55$ and $\dfrac{A'''-C'}{D'-B'''}>1.29$. 

Let us now check iv). In order to do this, we put $\lambda^+_0(u_{k-1,k-1})=[2;2_{2k-2}1\overline{21}]+[0;12_{2k-2}1\overline{12}]:=A^*+B^*$ and $\lambda_0(\theta(\underline{\omega}_k))>[2;2_{2k-2}12\overline{21}]+[0;12_{2k+1}12\overline{21}]:=C^*+D^*$. Note that 
$$\dfrac{D^*-B^*}{A^*-C^*} = \dfrac{q^2(2_{2k-2}12)}{q^2(12_{2k-2})}\cdot X^*\cdot Y^*$$ 
where 
$$X^*= \dfrac{[2;2212\overline{21}]-[1;\overline{12}]}{[2;\overline{12}]-[1;\overline{21}]}$$
and 
$$Y^*= \dfrac{([2;\overline{12}]+\beta(2_{2k-2}12))([1;\overline{21}]+\beta(2_{2k-2}12))}{([2;2212\overline{21}]+\beta(12_{2k-2}))([1;\overline{12}]+\beta(12_{2k-2}))}.$$
Since $\dfrac{q(2_{2k-2}12)}{q(12_{2k-2})} = 2+\beta(2_{2k-2}1)\geq 2+[0;12]>2.6$, $X^*>0.5$ and 
$$Y^*\geq \dfrac{([2;\overline{12}]+[0;2122])([1;\overline{21}]+[0;2122])}{([2;2212\overline{21}]+[0;221])([1;\overline{12}]+[0;221])}>0.87,$$ 
we deduce that $(D^*-B^*)/(A^*-C^*)>2.94>1$. 

Finally, let us verify v). For this sake, let us define $\lambda^+_0(u_{k,k-1}22)=[2;2_{2k-2}122\overline{21}]+[0;12_{2k}1\overline{12}]:=A^{**}+B^{**}$ and $\lambda_0(\theta(\underline{\omega}_k))>[2;2_{2k-2}12222\overline{12}]+[0;12_{2k+1}12\overline{21}]:=C^{**}+D^{**}$. Observe that 
$$\dfrac{D^{**}-B^{**}}{A^{**}-C^{**}} = \frac{q^2(2_{2k-2}1222)}{q^2(12_{2k})}\cdot X^{**}\cdot Y^{**}$$ 
where 
$$X^{**} = \dfrac{[2;12\overline{21}]-[1;\overline{12}]}{[2;\overline{12}]-[1;\overline{21}]}$$
and 
$$Y^{**} = \dfrac{([2;\overline{12}]+\beta(2_{2k-2}1222))([1;\overline{21}]+\beta(2_{2k-2}1222))}{([2;12\overline{21}]+\beta(12_{2k}))([1;\overline{12}]+\beta(12_{2k}))}.$$ 
Since $\dfrac{q(2_{2k-2}1222)}{q(12_{2k})} = \dfrac{17+12\beta(2_{2k-2})}{7+3\beta(2_{2k-2})}\geq \dfrac{17+12 [0;2222]}{7+3[0;222]}>2.6$, $X^{**}>0.71$ and 
$$Y^{**}\geq \dfrac{([2;\overline{12}]+[0;2221])([1;\overline{21}]+[0;2221])}{([2;12\overline{21}]+[0;221])([1;\overline{12}]+[0;221])}>0.82,$$ 
we conclude that $(D^{**}-B^{**})/(A^{**}-C^{**})>3.93>1$.
\end{proof}

\subsection{Ruling out $A_{a,b}$ with $a,b$ odd}  This case essentially never occurs.

\begin{lemma} \label{l.Aoddodd}
Let $u=12_{2j+1}12^*2_{2m+1}1$ with $2m+1\leq 2k-2$ and $2j+1\leq 2k+1$. If $m\le j$, resp. $j<m$, then $\lambda^+_0(u)<\lambda_0(\theta(\underline{\omega}_k))$, resp. $\lambda^-_0(u)>\lambda_0(\gamma^1_k)$. 
	
\end{lemma}

\begin{proof}
Let us first establish that $\lambda^+_0(u)<\lambda_0(\theta(\underline{\omega}_k))$ whenever $m\leq j$. For this purpose, we write $\lambda^+_0(u)=[2;2_{2m+1},1,\overline{1,2}]+[0;1,2_{2j+1},1,\overline{2,1}]:=A+B$ and $\lambda_0(\theta(\underline{\omega}_k))>[2;2_{2k-2},1,\overline{1,2}]+[0;1,2_{2k+1},1,\overline{1,2}]:=C+D$. If $j=k$, then we can apply Lemma \ref{cL.U1} to derive that $C+D>A+B$ because $C>A$, $B>D$ and $q(12_{2k+1}1)/q(2_{2m+1})>3$. If $j<k$, then 
$$\dfrac{C-A}{B-D} = \dfrac{q^2(12_{2j+1})}{q^2(2_{2m+1})}\cdot X\cdot Y$$ 
where 
$$X=\dfrac{[2;2_{2k-2m-4}1\overline{12}]-[1;\overline{12}]}{[2;2_{2k-2j-1}1\overline{12}]-[1;\overline{21}]}\geq \dfrac{[2;\overline{2}]-[1;\overline{12}]}{[2;\overline{2}]-[1;\overline{21}]}>0.65$$
and 
$$Y=\dfrac{([2;2_{2k-2j-1}1\overline{12}]+\beta(12_{2j+1}))([1;\overline{21}]+\beta(12_{2j+1}))}{([2;2_{2k-2m-4}1\overline{12}]+\beta(2_{2m+1}))([1;\overline{12}]+\beta(2_{2m+1}))}.$$
Since 
$$Y\geq \left\{\begin{array}{cc} \dfrac{([2;2,1,\overline{1,2}]+[0,2,2,2,1])([1;\overline{2,1}]+[0;2,2,2,1])}{([2;1,\overline{1,2}]+[0;2,2,2])([1;\overline{1,2}]+[0;2,2,2])}>0.773, & \textrm{if } m>0 \\ \dfrac{([2;2,1,\overline{1,2}]+[0,2,1])([1;\overline{2,1}]+[0;2,1])}{([2;2,2,\overline{1,2}]+[0;2])([1;\overline{1,2}]+[0;2])}>0.7, & \textrm{if } m=0\end{array}\right.$$
and 
$$\dfrac{q(12_{2j+1})}{q(2_{2m+1})}\geq 1+\beta(2_{2m+1})\geq \left\{\begin{array}{cc} 1+[0;\overline{2}], & \textrm{if } m>0 \\ 3/2, & \textrm{if } m=0\end{array}\right.$$ we see that $(C-A)/(D-B)>1.004$. 

Let us now show that $\lambda^-_0(u)>\lambda_0(\gamma^1_k)$ when $j<m$. In order to do this, we write $\lambda^-_0(u) = [2;2_{2m+1}1\overline{21}]+[0;12_{2j+1}1\overline{12}]:=B'+A'$ and $\lambda_0(\gamma^1_k)<[2;2_{2k-2}1\overline{21}]+[0;12_{2k+1}1\overline{21}]:=D'+C'$. Since $A'>C'$, $B'<D'$ and 
$$\dfrac{q(2_{2m+1})}{q(12_{2j+1})} \geq \dfrac{q(2_{2m+1})}{q(12_{2m-1})}=\dfrac{5+2\beta(2_{2m-1})}{1+\beta(2_{2m-1})}\geq \dfrac{5+2[0;22]}{1+[0;2]}>3.8,$$
we can use Lemma Lemma \ref{cL.U1} to conclude that $C'+D'<A'+B'$. 
\end{proof}

\subsection{Proof of Theorem \ref{t.local-uniqueness}} As it was said right before the statement of Theorem \ref{t.local-uniqueness}, a $(k,3.154)$-admissible word $\theta$ necessarily extends in one of the following ways: 
 
\begin{enumerate}
 
 \item[$A_{a,b}$:] $\theta=...12_a12^*2_b1...$ with $a\le 2k+1$ and $b\le 2k-2$,
 
 \item[$B_a$:] $\theta=...12_a12^*2_{2k-1}...$, with $a\leq 2k+1$.
 
 \item[$C_b$:] $\theta=...2_{2k+2}12^*2_b1...$ with $b\leq 2k-2$.
 
 \end{enumerate}

By Lemmas \ref{L.U1} and \ref{L.U2}, there is a constant $\lambda_k^{(1),B}>\lambda_0(\gamma_k^1)$ such that a $(k,\lambda_k^{(1),B})$-admissible word $\theta$ can not be of type $B_a$. Similarly, it follows from Lemmas \ref{L.U3}, \ref{L.U4}, \ref{l.Ck-1} and Lemma \ref{p1} that there exists a constant  $\lambda_k^{(1),C}>\lambda_0(\gamma_k^1)$ such that a $(k,\lambda_k^{(1),C})$-admissible word $\theta$ can not be of type $C_b$. Moreover, we have from Lemmas \ref{L.U5}, \ref{bpodd}, \ref{l.Aoddeveniii}, \ref{ooe}, \ref{L.U6}, \ref{L.U7}, \ref{l.Aoddodd} (together with Lemma \ref{p1}) that there is a constant $\lambda_k^{(1),A}>\lambda_0(\gamma_k^1)$ such that a $(k,\lambda_k^{(1),A})$-admissible word $\theta$ has the form $A_{2k+1,2k-2}$. This shows the validity of Theorem \ref{t.local-uniqueness} for $\lambda_k^{(1)}:=\min\{\lambda_k^{(1),A}, \lambda_k^{(1),B}, \lambda_k^{(1),C}\} >\lambda_0(\gamma_k^1)$. 

\subsection{The Markov values of $\theta(\underline{\omega}_k)$ and $\gamma_k^1$}

Closing this section, let us compute the Markov values of the sequences $\theta(\underline{\omega}_k)$ and $\gamma_k^1$. 

\begin{proposition}\label{p.3-1} For each $k\geq 3$, the Markov values of $\theta(\underline{\omega}_k)$ and $\gamma_k^1$ are attained at the position $0$. 
\end{proposition} 

\begin{proof} The Markov value of $\theta(\underline{\omega}_k)$ can be calculated as follows. Recall that 
$$\theta(\underline{\omega}_k) = \dots 12^*2_{2k-2}12_{2k}12_{2k+1}12_{2k-1}1\dots$$
By Lemma \ref{p1}, $\lambda_j(\theta(\underline{\omega}_k))<\lambda_0(\theta(\underline{\omega}_k))$ for all $j\neq 0, 2k-2, 2k, 4k-1, 4k+1, 6k+1$. Moreover, by Lemma \ref{L.U7} v), $\lambda_{2k-2}(\theta(\underline{\omega}_k))<\lambda_0(\theta(\underline{\omega}_k))$. Furthermore, by Lemma \ref{L.U2}, $\lambda_i(\theta(\underline{\omega}_k))<\lambda_0(\theta(\underline{\omega}_k))$ for $i=2k, 4k-1, 6k+1$. Also, by Lemma \ref{L.U1}, $\lambda_{4k+1}(\theta(\underline{\omega}_k))<\lambda_0(\theta(\underline{\omega}_k))$. This proves that $m(\theta(\underline{\omega}_k))=\lambda_0(\theta(\underline{\omega}_k))$. 

Similarly, the Markov value of $\gamma_k^1$ can be obtained in the following way. Recall that 
$$\gamma_k^1=\overline{2_{2k-1}12_{2k}12_{2k+1}1}2^*2_{2k-2}12_{2k}12_{2k+1}12_{2k-1}12_{2k}12_{2k-1}11\overline{2}$$ 
The arguments in the previous paragraph imply that $\lambda_{j}(\gamma_k^1)<\lambda_0(\theta(\underline{\omega}_k))<\lambda_0(\gamma_k^1)$ for all $j\notin -(6k+4)\mathbb{N}^*\cup \{6k+3, 8k+1, 10k+4, 12k+2\}$. Also, a direct comparison shows that $\lambda_i(\gamma_k^1)<\lambda_0(\gamma_k^1)$ for each $i\in -(6k+4)\mathbb{N}^*\cup \{6k+3, 8k+1, 10k+4, 12k+2\}$. This completes the proof of the proposition. 
\end{proof} 

\section{Going for the replication}\label{s.extension} 

In this section, we investigate for every $k \geq 4$ the extensions of a word $\theta$ containing the string 
$$\alpha_k^1=12_{2k+1}12^*2_{2k-2}1.$$ 

More concretely, the main result of this section is the following statement:

\begin{theorem}\label{t.extension} For each $k\geq 4$, there is an explicit constant $\mu_k^{(1)}>\lambda_0(\gamma_k^1)$ such that any $(k,\mu_k^{(1)})$-admissible word $\theta$ containing $\alpha_k^1$ extends as 
 $$\theta =...2_{2k+1}12_{2k-1}12_{2k}12_{2k+1}12^*2_{2k-2}12_{2k}12_{2k+1}12_{2k-1}....$$ 
\end{theorem}

Once again, the proof of this theorem will take this entire section. 

\subsection{Extension from $\alpha_k^1$ to $2_{2k}\alpha_k^12_{2k}$ }
\begin{lemma} \label{LG1} Let $\alpha_k^1=12_{2k+1}12^*2_{2k-2}1$ with $k\geq 3$. We have: 
\begin{itemize}
\item[i)] $\lambda^+_0(\alpha_k^11)<\lambda^+_0(\alpha_k^12_21)<m(\theta(\underline{\omega}_k));$
\item[ii)]$\lambda^+_0(1\alpha_k^12222)<m(\theta(\underline{\omega}_k));$
\end{itemize}
\end{lemma}
\begin{proof}
 Note that $[2;2_{2k-2},1,1,...]<[2;2_{2k-2},1,2,2,1,...]$. In particular, $\lambda^+_0(\alpha^1_k1)<\lambda^+_0(\alpha^1_k221)$. 
To prove that $\lambda^+_0(\alpha^1_k221)<m(\theta(\underline{\omega}_k))$, we can use Lemma \ref{cL.U1} with $\underline{a}=2_{2k-2}12_2$ and $\underline{b}=12_{2k+1}12$. In fact, observe that
	$$\lambda^+_0(\alpha^1_k221)=[2;2_{2k-2},1,2_2,1,\overline{1,2}]+[0;1,2_{2k+1},1,\overline{2,1}]:=A+B$$
and
	$$m(\theta(\underline{\omega}_k))>[2;2_{2k-2},1,2_{2k},1,\overline{2,1}]+[0;1,2_{2k+1},1,2_{2k},1,\overline{2,1}]:=C+D.$$
We have $C>A$ and $B>D$. Moreover, by Euler's rule,  
	$$q(\underline{b})=q(\underline{b}^t)>q(212_3)q(2_{2k-2}1)=46 q(2_{2k-2}1)$$
and 
	$$q(\underline{a})=q(\underline{a}^t)=q(2_2)q(2_{2k-2}1)+q(2)q(2_{2k-2})<7q(2_{2k-2}1).$$
	This implies that 
		$$q(\underline{b})>4q(\underline{a})$$
and, hence, $C+D>A+B$ thanks to Lemma \ref{cL.U1}. This completes the proof of i). 

To prove ii) we write $\lambda^+_0(1\alpha^1_k2_4)=[2;2_{2k-2},1,2_4,\overline{2,1}]+[0;1,2_{2k+1},1,1,\overline{1,2}]:=A'+B'$ and $m(\theta(\underline{\omega}_k))>C+D$ as above. By Euler's rule 
$$\dfrac{q(2_{2k-2}12_5)}{q(12_{2k+1}1)}=\dfrac{99+70\beta(2_{2k-2})}{24+10\beta(2_{2k-2})}>4,$$ so that 
$A'+B'<C+D$ thanks to Lemma \ref{cL.U1}.
\end{proof}

Since the word $12^*1$ is $k$-prohibited, it follows from Lemma \ref{LG1} that $\alpha^1_k$ must be continued as $\alpha_k^1 2_3$. Furthermore, by Lemma \ref{bpodd} and Lemma \ref{LG1} ii), we must continue $\alpha^1_k2_3$ as $2_2\alpha^1_k2_4$. In summary, we have:
\begin{corollary}\label{c.Ext1-0}
Consider the parameter
$$\lambda_k^{(2)}:=\lambda_0^-(2_{2k-2}12^*221).$$
Then, $\lambda_k^{(2)}>m(\gamma_k^1)$ and any $(k,\lambda_k^{(2)})$-admissible word $\theta$ containing $\alpha_k^1$ extends as $$\theta=...2_2\alpha_k^12_4...=...2_212_{2k+1}12^*2_{2k-2}12_4....$$
\end{corollary}

In general, the word $\theta=...2_2\alpha_k^12_4$ continues as $\theta=...2_{a}\alpha_k^1 2_b...$ with $a\geq 2$ and $b\geq 4$. If $a,b>2k$, then $\lambda_0^-(\theta)>m(\gamma_k^1)$. Thus, we have four cases:
\begin{itemize}
\item[Ext1A)] The string $2_{2k}{\alpha}_k^12_{2k}$.

\item[Ext1B)] The string $\gamma_{a,b}=12_a{\alpha}_k^12_b1$, with $a,b<2k$.

\item[Ext1C)] The string $\gamma_b=2_{2k}{\alpha}_k^12_b1$, with $b<2k$.

\item[Ext1D)] The string $\gamma^a=12_a{\alpha}_k^12_{2k}$, with $a<2k$.
\end{itemize}

\subsubsection{Ruling out Ext1B)} This case essentially never occurs. In order to see this, let $\gamma_{a,b}=12_a{\alpha}_k^12_b1=12_{a}12_{2k+1}12^*2_{2k-2}12_{b}1$. We have the following subcases:
\begin{itemize}
\item[Ext1B1)] $b$ odd and $a$ odd;
\item[Ext1B2)] $b$ odd and $a$ even;
\item[Ext1B3)] $b$ even and $a$ odd;
\item[Ext1B4)] $b$ even and $a$ even.
\end{itemize}
The next lemma asserts that the case Ext1B1) essentially never occurs:
\begin{lemma}
If $a=2j+1<2k$ and $b=2m+1<2k$, then $\lambda_0^-(\gamma_{a,b})\ge\lambda_0^-(\gamma_{2k-1,2k-1})>m(\gamma_k^1)$.
\end{lemma} 
\begin{proof}
For $a=2j+1<2k$ and $b=2m+1<2k$, the inequality $\lambda_0^-(\gamma_{a,b})\ge\lambda_0^-(\gamma_{2k-1,2k-1})$ is straightforward. Hence, it remains to prove that  $\lambda_0^-(\gamma_{2k-1,2k-1})>m(\gamma_k^1)$. For this sake, note that:
$$A:=[2;2_{2k-2},1,2_{2k-1},1,\overline{1,2}]>[2;2_{2k-2},1,2_{2k},1,2,\overline{2,1}]=:C\; {\rm and}$$
$$B:=[0;1,2_{2k+1},1,2_{2k-1},1,\overline{1,2}]>[0;1,2_{2k+1},1,2_{2k},1,2,\overline{2,1}]=:D.$$
Therefore, $\lambda_0^-(\gamma_{2k-1,2k-1}):=A+B>C+D>m(\gamma_k^1)$.
\end{proof}

The case Ext1B2) essentially never occurs. Indeed, first note that in this setting ($b=2m+1<2k$ is odd) one actually has $b=2k-1$ by Lemma \ref{bpodd}. Also, note that $\lambda^-_0(\gamma_{2j,2k-1})$ and $\lambda^+_0(\gamma_{2j,2k-1})$ 
are increasing functions of $j$. In particular, $\lambda_0^-(\gamma_{2k-2,2k-1})>\lambda_0^-(\gamma_{2k-4,2k-1})$ and $\lambda_0^+(\gamma_{2j,2k-1})\le\lambda_0^+(\gamma_{2k-6,2k-1})$ for all $2j \le 2k-6$. Thus, we can rule out Ext1B2) using the following lemma: 
\begin{lemma} We have:
\begin{itemize}
\item[i)] $\lambda_0^-(\gamma_{2k-4,2k-1})>m(\gamma_k^1)$;
\item[ii)] $\lambda_0^+(\gamma_{2k-6,2k-1})<m(\theta(\underline{\omega}_k))$.
\end{itemize}
\end{lemma}
\begin{proof}
To prove i) we write  
	$$\lambda^-_0(\gamma_{2k-4,2k-1})=[2;2_{2k-2},1,2_{2k-1},1,\overline{1,2}]+[0;1,2_{2k+1},1,2_{2k-4},1,\overline{2,1}]:=A+B.$$
and
	$$m(\gamma^1_k)<[2;2_{2k-2},1,2_{2k},1,2_2,\overline{1,2}]+[0;1,2_{2k+1},1,2_{2k},1,2_2,\overline{1,2}]:=C+D.$$
Therefore,
	$$A-C=\dfrac{[2;1,2_2,\overline{1,2}]-[1;\overline{1,2}]}{q^2(2_{2k-2}12_{2k-1})([2;1,2_2,\overline{1,2}]+\beta)([1;\overline{1,2}]+\beta)},$$
where $\beta=\beta(2_{2k-2},1,2_{2k-1})=[0;2_{2k-1},1,2_{2k-2}]<[0;\overline{2}].$ Moreover, we have
	$$D-B=\dfrac{[2;2_{3},1,2_2,\overline{1,2}]-[1;\overline{2,1}]}{q^2(12_{2k+1}12_{2k-4})([2;2_{3},1,2_2,\overline{1,2}]+\tilde{\beta})([1;\overline{2,1}]+\tilde{\beta})},$$
where $\tilde{\beta}=\beta(12_{2k+1}12_{2k-4})=[0;2_{2k-4},1,2_{2k+1},1]>[0;\overline{2}].$ In particular,
	$$\dfrac{A-C}{D-B}=\dfrac{q^2(12_{2k+1}12_{2k-4})}{q^2(2_{2k-2}12_{2k-1})}\cdot X \cdot Y,$$
where
	$$X=\dfrac{[2;1,2_2,\overline{1,2}]-[1;\overline{1,2}]}{[2;2_{3},1,2_2,\overline{1,2}]-[1;\overline{2,1}]}>0.927$$
and $$Y=\dfrac{([2;2_{3}12_2\overline{12}]+\tilde{\beta})([1;\overline{21}]+\tilde{\beta})}{([2;12_2\overline{12}]+\beta)([1;\overline{12}]+\beta)}>\dfrac{([2;2_{3}12_2\overline{12}]+[0;\overline{2}])([1;\overline{21}]+[0;\overline{2}])}{([2;12_2\overline{12}]+[0;\overline{2}])([1;\overline{12}]+[0;\overline{2}])}>0.752.$$ 
Also, by Euler's rule, 
\begin{eqnarray*}\frac{q(12_{2k+1}12_{2k-4})}{q(2_{2k-2}12_{2k-1})} &=& \frac{7q(2_{2k-4}12_{2k-1})+3q(2_{2k-4}12_{2k-2})}{2q(2_{2k-1}12_{2k-3})+q(2_{2k-1}12_{2k-4})}>\frac{7\beta(2_{2k-1}12_{2k-3})}{2+\beta(2_{2k-1}12_{2k-3})} \\ &>&\frac{7}{\frac{2}{[0;2_4]}+1}=1.2.
\end{eqnarray*} 
Thus, 
	$$\dfrac{A-C}{D-B}>(1.2)^2\cdot 0.927 \cdot 0.752>1.003.$$
	
The proof of ii) follows from Lemma \ref{cL.U1} because 
\begin{eqnarray*} \frac{q(2_{2k-2}12_{2k-1})}{q(12_{2k+1}12_{2k-6})} = \frac{29q(2_{2k-1}12_{2k-6})+12q(2_{2k-1}12_{2k-7})}{3q(2_{2k-6}12_{2k})+q(2_{2k-6}12_{2k-1})}>\frac{29}{\frac{3}{\beta(2_{2k-6}12_{2k})}+1}>3.5
\end{eqnarray*}
thanks to Euler's rule. 
\end{proof}

The case Ext1B3) essentially never occurs. In fact, note that in this context ($a=2j+1<2k$ is odd), we can apply Lemma \ref{bpodd} to assume that $a=2k-1$. The following lemma asserts that this possibility doesn't occur:
\begin{lemma}
If $b=2m \le 2k-2$, then $\lambda_0^+(\gamma_{2k-1,2m})\le\lambda_0^+(\gamma_{2k-1,2k-2})<m(\theta(\underline{\omega}_k))$.
\end{lemma}
\begin{proof}
First, we have the inequality $\lambda_0^+(\gamma_{2k-1,2m})\le\lambda_0^+(\gamma_{2k-1,2k-2})$ for every $b=2m \le 2k-2$.  

Thus, it remains prove that $\lambda_0^+(\gamma_{2k-1,2k-2})<m(\theta(\underline{\omega}_k))$. This estimate follows from Lemma \ref{cL.U1} because 
$$\lambda_0^+(\gamma_{2k-1,2k-2})=[2;2_{2k-2},1,2_{2k-2},1,\overline{1,2}]+[0;1,2_{2k+1},1,2_{2k-1},1,\overline{2,1}]:=C+D,$$
$$m(\theta(\underline{\omega}_k))>[2;2_{2k-2},1,2_{2k},1,2,\overline{1,2}]+[0;1,2_{2k+1},1,2_{2k},1,2,\overline{1,2}]:= A+B,$$ 
and 
\begin{eqnarray*}
\frac{q(12_{2k+1}12_{2k-1})}{q(2_{2k-2}12_{2k-2})} &=& \frac{2q(12_{2k+1}12_{2k-2})+q(12_{2k+1}12_{2k-3})}{q(2_{2k-2}12_{2k-2})}\ge \left(2+\frac{1}{3}\right) \frac{q(12_{2k+1}12_{2k-2})}{q(2_{2k-2}12_{2k-2})} \\ 
&\geq&\frac{7}{3}q(2_31)>3
\end{eqnarray*} 
thanks to Euler's rule. 
\end{proof}

Finally, a direct comparison of continued fractions reveals that the case Ext1B4) essentially never occurs.
\begin{lemma}
If $a=2j<2k$ and $b=2m<2k$, then $\lambda_0^+(\gamma_{2j,2m})\le\lambda_0^+(\gamma_{2k-2,2k-2})<m(\theta(\underline{\omega}_k))$.
\end{lemma}
\begin{proof}
Note that  
	$$[2;2_{2k-2},1,2_{2m},1,...]\le [2;2_{2k-2},1,2_{2k-2},1,...]<[2;2_{2k-2},1,2_{2k},1...] $$
and
	 $$[0;1,2_{2k+1},1,2_{2j},1,..]\le [2;2_{2k-2},1,2_{2k-2},1,...]<[0;1,2_{2k+1},1,2_{2k},1,...]$$
whenever $j,m<k$. 	 
\end{proof}

\subsubsection{Ruling out Ext1C)}
We begin by excluding Ext1C) with $b$ odd: 
\begin{lemma}\label{e2.c.o}
If $0<m\le k-1$ and $u_m=2_{2k}{\alpha_k^1}2_{2m+1}1$ then 
$$\lambda^-_0(u_m)\ge \lambda^-_0(u_{k-1})>m(\gamma^1_k).$$
\end{lemma}

\begin{proof}
We write 
$$\lambda^-_0(u_{k-1})=[2;2_{2k-2},1,2_{2k-1},1,\overline{1,2}]+[0;1,2_{2k+1},1,2_{2k},\overline{1,2}]:=A+B$$
and
$$m(\gamma^1_k)<[2;2_{2k-2},1,2_{2k},1,\overline{1,2}]+[0;1,2_{2k+1},1,2_{2k},1,\overline{1,2}]:=C+D.$$
Then $A>C$ and $D>B$. By Lemma \ref{cL.U1}, it follows that
	$$A+B>C+D$$ 
since $q(12_{2k+1}12_{2k}1)>4\cdot q(2_{2k-2}12_{2k-1})$. 
\end{proof}

Let us now exclude Ext1C) with $b$ even: 

\begin{lemma}
If $m<k$ and $u_m=2_{2k}12_{2k+1}12^*2_{2k-2}12_{2m}1$ then
	$$\lambda^+_0(u_m)\le \lambda^+_0(u_{k-1})<m(\theta(\underline{\omega}_k)).$$
\end{lemma}
\begin{proof}
The proof is similar to the proof of Lemma \ref{e2.c.o}. Just note that now $C>A$ and $B>D$ and, by Lemma \ref{cL.U1},  $A+B<C+D$.
\end{proof}
	
\subsubsection{Ruling out Ext1D)} Let us first show that Ext1D) with $a$ even essentially never occurs. For this sake, we use the Lemma \ref{p1} i) and the next two lemmas:
\begin{lemma}
Let $\gamma^a=12_a{\alpha}_k^12_{2k}=12_{a}12_{2k+1}12^*2_{2k-2}12_{2k}$.  If $a=2j \le 2k-4$, then $\lambda_0^+(\gamma^{2j})\le\lambda_0^+(\gamma^{2k-4})<m(\theta(\underline{\omega}_k))$.
\end{lemma}
\begin{proof}
First, we have that $\lambda_0^+(\gamma^{2j})\le\lambda_0^+(\gamma^{2k-4})$, for every $a=2j \le 2k-4$. 

Let $\lambda^{+}_0(\gamma^{2k-4})=[2;2_{2k-2},1,2_{2k},\overline{2,1}]+[0;1,2_{2k+1},1,2_{2k-4},1,\overline{1,2}]:=C+D$ and  $m(\theta(\underline{\omega}_k))>[2;2_{2k-2},1,2_{2k},1,2,2,\overline{2,1}] + [0;1,2_{2k+1},1,2_{2k},1,2,2,\overline{2,1}]:= A+B$. Our task is reduced to prove that $B-D>C-A$. In order to establish this inequality, we observe that
$$B-D=\dfrac{[2;2_3,1,2,2,\overline{2,1}]-[1;\overline{1,2}]}{{q}_{4k-1}^2([2;2_3,1,2,2,\overline{2,1}]+{\beta})([1;\overline{1,2}]+{\beta})},$$
and
$$C-A=\dfrac{[2;\overline{1,2}]-[1;2,2,\overline{2,1}]}{\tilde{q}_{4k-1}^2([2;\overline{1,2}]+\tilde{\beta})([1;2,2,\overline{2,1}]+\tilde{\beta})}$$
where $q_{4k-1}=q(12_{2k+1}12_{2k-4})$, $\tilde{q}_{4k-1}=q(2_{2k-2}12_{2k})$, $\beta=[0;2_{2k-4},1,2_{2k+1},1]$ and $\tilde{\beta}=[0;2_{2k},1,2_{2k-2}]$.
Thus,
$$\dfrac{B-D}{C-A}=\dfrac{[2;2_3,1,2,2,\overline{2,1}]-[1;\overline{1,2}]}{[2;\overline{1,2}]-[1;2,2,\overline{2,1}]}\cdot Y \cdot \dfrac{\tilde{q}^2_{4k-1}}{{q}^2_{4k-1}} > 0.51\cdot Y \cdot \dfrac{\tilde{q}^2_{4k-1}}{{q}^2_{4k-1}},$$
where
$$Y=\dfrac{([2;\overline{1,2}]+\tilde{\beta})([1;2,2,\overline{2,1}]+\tilde{\beta})}{([2;2_3,1,2,2,\overline{2,1}]+{\beta})([1;\overline{1,2}]+{\beta})}.$$
Note that
$$Y>\dfrac{([2;\overline{1,2}]+[0,\overline{2}])([1;2,2,\overline{2,1}]+[0,\overline{2}])}{([2;2_3,1,2,2,\overline{2,1}]+[0,2_4,1])([1;\overline{1,2}]+[0,2_4,1])}>0.94.$$
Let $\Gamma=2_{2k}12_{2k-4}$. By Euler's rule and Lemma \ref{bi} i), we have:
\begin{align*}
q_{4k-1}=q(2_{2k-4}12_{2k+1})+q(\Gamma^t)=3q(\Gamma^t)+q(2_{2k-4}12_{2k-1})<(3+1/2)q(\Gamma^t)
\end{align*} 
and  
\begin{align*}
\tilde{q}_{4k-1} = 2q(2_{2k}12_{2k-3})+q(2_{2k}12_{2k-4})= 5 q(\Gamma)+2q(2_{2k}12_{2k-5})>q(\Gamma)(5+2/3).
\end{align*} 
Thus, 
\begin{align*}
\dfrac{\tilde{q}_{4k-1}}{q_{4k-1}}  > \dfrac{34}{21}.
\end{align*}
Therefore, $\dfrac{B-D}{C-A} > 0.51 \cdot 0.94 \cdot \left( \dfrac{34}{21}\right)^2>1.25>1$. 
\end{proof}

\begin{lemma} Let $\gamma^{2k-2}=12_{2k-2}{\alpha}_k^12_{2k}=12_{2k-2}12_{2k+1}12^*2_{2k-2}12_{2k}$. We have:
\begin{itemize}
\item[i)] $\lambda_0^-(\gamma^{2k-2}2)>\lambda_0^-(\gamma^{2k-2}11)>m(\gamma_k^1)$;
\item[ii)] $\lambda_0^+(\gamma^{2k-2}122)<m(\theta(\underline{\omega}_k))$.
\end{itemize}
\end{lemma}
\begin{proof}
In order to prove i) we first note that $[2;2_{2k-2},1,2_{2k},2,...]>[2;2_{2k-2},1,2_{2k},1...]$ 
and, hence, $\lambda_0^-(\gamma^{2k-2}2)>\lambda_0^-(\gamma^{2k-2}11)$. Next, we write 
$$\lambda_0^-(\gamma^{2k-2}11)=[2;2_{2k-2},1,2_{2k},1_2,\overline{1,2}]+[0;1,2_{2k+1},1,2_{2k-2},1,\overline{2,1}]:=A+B$$
and 
	$$m(\gamma^1_k)<[2;2_{2k-2},1,2_{2k},1,2_2,\overline{1,2}]+[0;1,2_{2k+1},1,2_{2k-1},\overline{1,2}]:=C+D.$$
Note that $A>C$, $D>B$ and 
$$\frac{A-C}{D-B} = \dfrac{[2;2,\overline{1,2}]-[1;\overline{1,2}]}{[2;\overline{1,2}]-[1;\overline{1,2}]}\cdot Y \cdot \dfrac{q^2(12_{2k+1}12_{2k-2})}{q(2_{2k-2}12_{2k}1)} > 0.63\cdot Y \cdot \dfrac{q^2(12_{2k+1}12_{2k-2})}{q(2_{2k-2}12_{2k}1)}$$ 
where 
\begin{eqnarray*}
Y&=&\dfrac{([2;\overline{1,2}]+\beta(12_{2k+1}12_{2k-2}))([1;\overline{1,2}]+\beta(12_{2k+1}12_{2k-2}))}{([2;2,\overline{1,2}]+\beta(2_{2k-2}12_{2k}1))([1;\overline{1,2}]+\beta(2_{2k-2}12_{2k}1))} \\ 
&>&\frac{([2;\overline{12}]+[0;\overline{2}])([1;\overline{1,2}]+[0;\overline{2}])}{([2;2,\overline{1,2}]+[0;1,\overline{2}])([1;\overline{1,2}]+[0;1,\overline{2}])}>0.9
\end{eqnarray*}
Since
$$q(12_{2k+1}12_{2k-2})=3q(2_{2k-2}12_{2k})+q(2_{2k-2}12_{2k-1})$$
and 
$$q(2_{2k-2}12_{2k}1) = q(2_{2k-2}12_{2k})+q(2_{2k-2}12_{2k-1}),$$
we also have that 
$$\frac{q(12_{2k+1}12_{2k-2})}{q(2_{2k-2}12_{2k}1)} = \frac{3+\beta(2_{2k-2}12_{2k})}{1+\beta(2_{2k-2}12_{2k})} > 2.41.$$ Therefore, $(A-C)/(D-B)>1$. 

	To prove ii), it suffices to apply Lemma \ref{cL.U1}. In fact, we can write
	$$\lambda^+_0(\gamma^{2k-2}122)=[2;2_{2k-2},1,2_{2k},1,2_2,\overline{1,2}]+[0;1,2_{2k+1},1,2_{2k-2},1,\overline{1,2}]:=D'+C'$$
	and
	$$m(\theta(\underline{\omega}_k))>[2;2_{2k-2},1,2_{2k},1,2_{4},\overline{2,1}]+[0;1,2_{2k+1},1,2_{2k}1,\overline{2,1}]:=B'+A',$$ 
with $B'>D'$, $C'>A'$ and $q(2_{2k-2}12_{2k}12_2)>4\cdot q(12_{2k+1}12_{2k-2})$. 
\end{proof}

Now, let us prove that Ext1D) with $a$ odd essentially never occurs. In this regime ($a=2j+1<2k$ is odd), Lemma \ref{bpodd} says that we can assume that $a=2k-1$. So, we can exclude Ext1D) with $a$ odd thanks to Lemma \ref{p1} i) and the next lemma:
\begin{lemma} \label{e2.d.o}
Let $\gamma^{2k-1}=12_{2k-1}{\alpha}_k^12_{2k}=12_{2k-1}12_{2k+1}12^*2_{2k-2}12_{2k}$. Then, $\lambda_0^-(\gamma^{2k-1}2)>\lambda_0^-(\gamma^{2k-1}11)>\lambda_0^-(\gamma^{2k-1}122)>m(\gamma_k^1)$.
\end{lemma}
\begin{proof}
First, by parity we check that $\lambda_0^-(\gamma^{2k-1}2)>\lambda_0^-(\gamma^{2k-1}11)>\lambda_0^-(\gamma^{2k-1}122)$. It remains to prove that  
$\lambda_0^-(\gamma^{2k-1}122)>m(\gamma_k^1)$. We write 
$$\lambda_0^-(\gamma^{2k-1}122):=C+D:=[2;2_{2k-2},1,2_{2k},1,2_2,\overline{2,1}]+[0;1,2_{2k+1},1,2_{2k-1},1,\overline{1,2}]$$
and
$$m(\gamma_k^1)<A+B:=[2;2_{2k-2},1,2_{2k},1,2_4,\overline{1,2}]+[0;1,2_{2k+1},1,2_{2k},1,2_2,\overline{1,2}],$$
so that our task is reduced to prove that $D-B>A-C$. 

Observe that 
$$D-B=\dfrac{[2;1,2,2,\overline{1,2}]-[1;\overline{1,2}]}{{q}_{4k+2}^2([2;1,2,2,\overline{1,2}]+{\beta})([1;\overline{1,2}]+{\beta})},$$
and
$$A-C=\dfrac{[1;2_4\overline{1,2}]-[1;2,2,\overline{2,1}]}{\tilde{q}_{4k-1}^2([1;2_4,\overline{1,2}]+\tilde{\beta})([1;2,2,\overline{2,1}]+\tilde{\beta})}$$
where $q_{4k+2}=q(12_{2k+1}12_{2k-1})$, $\tilde{q}_{4k-1}=q(2_{2k-2}12_{2k})$, $\beta=[0;2_{2k-1},1,2_{2k+1},1]$ and $\tilde{\beta}=[0;2_{2k},1,2_{2k-2}]$.
Thus,
$$\dfrac{D-B}{A-C}=\dfrac{[2;1,2,2,\overline{1,2}]-[1;\overline{1,2}]}{[1;2_4,\overline{1,2}]-[1;2,2,\overline{2,1}]}\cdot Y \cdot \dfrac{\tilde{q}^2_{4k-1}}{{q}^2_{4k+2}} > 574.47\cdot Y \cdot \dfrac{\tilde{q}^2_{4k-1}}{{q}^2_{4k+2}},$$
where
$$Y=\dfrac{([1;2_4,\overline{1,2}]+\tilde{\beta})([1;2,2,\overline{2,1}]+\tilde{\beta})}{([2;1,2,2,\overline{1,2}]+{\beta})([1;\overline{1,2}]+{\beta})}.$$
Note that
$$Y>\dfrac{([1;2_4,\overline{1,2}]+[0;\bar{2}])([1;2,2,\overline{2,1}]+[0;\bar{2}])}{([2;1,2,2,\overline{1,2}]+[0;\bar{2}])([1;\overline{1,2}]+[0;\bar{2}])}>0.5.$$
Let $\Gamma=2_{2k-2}12_{2k}$, by Euler's rule and Lemma \ref{bi}$i)$, we have:
\begin{align*}
q_{4k+2}&=2q(12_{2k+1}12_{2k-2})+q(12_{2k+1}12_{2k-3})<\left(2+\frac{1}{2}\right)q(12_{2k+1}12_{2k-2})= \\
&=\frac{5}{2}[q(2_{2k-2}12_{2k+1})+q(\Gamma)]=\frac{5}{2}[3q(\Gamma)+q(2_{2k-2}12_{2k-1})]=\frac{5}{2}\left(3+\frac{1}{2}\right)q(\Gamma)
\end{align*} 

Thus, $\dfrac{\tilde{q}_{4k-1}}{q_{4k+2}}  > \dfrac{4}{35}$ and, therefore,
$\dfrac{D-B}{A-C}> 574.47 \cdot 0.5 \cdot \left( \dfrac{4}{35}\right)^2>3.75>1$. 
\end{proof}

\subsubsection{Conclusion: Ext1B), Ext1C), Ext1D) are ruled out}

Our discussion after Corollary \ref{c.Ext1-0} until now implies that Ext1A) is essentially the sole possible extension of $\theta=2_2\alpha_k^12_4$: in fact, we have proved that 

\begin{corollary}\label{c.Ext1-1}
There exists an explicit parameter $\lambda_k^{(3)}>m(\gamma_k^1)$ such that any $(k,\lambda_k^{(3)})$-admissible word $\theta$ containing $2_2\alpha_k^12_4$ extends as $$\theta=...2_{2k}\alpha_k^12_{2k}=...2_{2k}12_{2k+1}12^*2_{2k-2}12_{2k}....$$
\end{corollary}

\subsection{Extension from $2_{2k}\alpha_k^12_{2k}$ to $2_{2k-1}12_{2k}\alpha_k^12_{2k}12_{2k+1}$}

\begin{lemma}\label{te2-3.1}
$\lambda_0^-(2_{2k}\alpha_k^12_{2k}2)>\lambda_0^-(2_{2k}\alpha_k^12_{2k}11)>\lambda_0^-(2_{2k}\alpha_k^12_{2k}1221)>m(\gamma_k^1)$.
\end{lemma}

\begin{proof}
It is not hard to see that $\lambda_0^-(2_{2k}\alpha_k^12_{2k}2)>\lambda_0^-(2_{2k}\alpha_k^12_{2k}11)>\lambda_0^-(2_{2k}\alpha_k^12_{2k}1221)$: just observe that
	$$[0;2_{2k-2},1,2_{2k},2,...]>[0;2_{2k-2},1,2_{2k},1,1,...]>[0;2_{2k-2},1,2_{2k},1,2,2,1,...].$$
	In order to prove that $\lambda_0^-(2_{2k}\alpha_k^12_{2k}1221)>m(\gamma_k^1)$, we write
	$$\lambda_0^-(2_{2k}\alpha_k^12_{2k}1221)=[2;2_{2k-2},1,2_{2k},1,2_2,1,\overline{1,2}]+[0;1,2_{2k+1},1,2_{2k},\overline{1,2}]:=A+B$$
	and
	$$m(\gamma^1_k)<[2;2_{2k-2},1,2_{2k},1,2_{2k+1},1,\overline{1,2}]+[0;1,2_{2k+1},1,2_{2k},1,2_{2k-1},1,\overline{1,2}]:=C+D$$
	Since $q(2_{2k-2}12_{2k}12_2)<3\cdot q(2_{2k-2}12_{2k}12)$ and
	$$q(12_{2k+1}12_{2k}12)>q(12_{3})q(2_{2k-2}12_{2k}12)>17\cdot q(2_{2k-2}12_{2k}12),$$
	we have $q(12_{2k+1}12_{2k}12)>4\cdot q(2_{2k-2}12_{2k}12_2)$. Because $A>C$ and $D>B$, it follows from Lemma \ref{cL.U1} that $A+B>C+D$. 
\end{proof}

\begin{lemma}\label{te2-3.2}
$\lambda_0^-(22_{2k}\alpha_k^12_{2k}12_4)>\lambda_0^-(112_{2k}\alpha_k^12_{2k}12_4)>m(\gamma_k^1)$.
\end{lemma}
\begin{proof}
By direct inspection, we see that 
	$$\lambda_0^-(22_{2k}\alpha_k^12_{2k}12_4)>\lambda_0^-(112_{2k}\alpha_k^12_{2k}12_4).$$
	It remains to prove that $\lambda_0^-(112_{2k}\alpha_k^12_{2k}12_4)>m(\gamma_k^1)$. 
In order to prove this inequality, let  
$$\lambda_0^-(112_{2k}\alpha_k^12_{2k}12_4)=[2;2_{2k-2},1,2_{2k},1,2_4,\overline{2,1}]+[0;1,2_{2k+1},1,2_{2k},1,1,\overline{1,2}]:=C+D$$ 
and 
$$m(\gamma_k^1)<[2;2_{2k-2},1,2_{2k},1,2_6,\overline{1,2}]+[0;1,2_{2k+1},1,2_{2k},1,2,2,\overline{1,2}]:=A+B.$$ 
Our task is reduced to prove that $D-B>A-C$. We have:
$$D-B=\dfrac{[2;2,\overline{1,2}]-[1;\overline{1,2}]}{\tilde{q}_{4k+4}^2([2;2,\overline{1,2}]+\tilde{\beta})([1;\overline{1,2}]+\tilde{\beta})},$$
and
$$A-C=\dfrac{[2;2_3,\overline{2,1}]-[2;2_5,\overline{1,2}]}{{q}_{4k}^2([2;2_3,\overline{2,1}]+{\beta})([2;2_5,\overline{1,2}]+{\beta})}$$
where $q_{4k}=q(2_{2k-2}12_{2k}1)$, $\tilde{q}_{4k+4}=q(12_{2k+1}12_{2k}1)$, $\beta=[0;1,2_{2k},1,2_{2k-2}]$ and $\tilde{\beta}=[0;1,2_{2k},1,2_{2k+1},1]$.
Thus,
$$\dfrac{D-B}{A-C}=\dfrac{[2;2,\overline{1,2}]-[1;\overline{1,2}]}{[2;2_3,\overline{2,1}]-[2;2_5,\overline{1,2}]}\cdot Y \cdot \dfrac{{q}^2_{4k}}{\tilde{q}^2_{4k+4}} > 2185.35 \cdot Y \cdot \dfrac{{q}^2_{4k}}{\tilde{q}^2_{4k+4}},$$
where
$$Y=\dfrac{([2;2_3,\overline{2,1}]+{\beta})([2;2_5,\overline{1,2}]+{\beta})}{([2;2,\overline{1,2}]+\tilde{\beta})([1;\overline{1,2}]+\tilde{\beta})}.$$
Note that
$$Y>\dfrac{([2;2_3,\overline{2,1}]+[0;1,2_5])([2;2_5,\overline{1,2}]+[0;1,2_5])}{([2;2,\overline{1,2}]+[0;1,\bar{2}])([1;\overline{1,2}]+[0;1,\bar{2}])}>1.29.$$
Also, by Euler's rule, we have:
\begin{align*}
\tilde{q}_{4k+4}&=q(12_{2k}12_{2k+1}1)<2q(12_{2k}12_{2k-2})q(2_31)=2\cdot q_{4k}\cdot 17
\end{align*} 
Therefore,
$$\dfrac{D-B}{A-C}> 2185.35 \cdot 1.29 \cdot \left( \dfrac{1}{34}\right)^2>2.43>1.$$
\end{proof}

As a direct consequence of the previous two lemmas and Corollary \ref{c.Ext1-0}, we get: 

\begin{corollary}\label{c.Ext2-0}
Consider the parameter
$$\lambda_k^{(4)}:=\min\{\lambda_0^-(2_{2k}\alpha_k^12_{2k}1221),\lambda_0^-(112_{2k}\alpha_k^12_{2k}12_4),\lambda_0^-(2_{2k-2}12^*2_{2}1):=\lambda_k^{(2)}\}.$$
Then, $\lambda_k^{(4)}>m(\gamma_k^1)$ and any $(k,\lambda_k^{(4)})$-admissible word $\theta$ containing $2_{2k}\alpha_k^12_{2k}$ extends as $$\theta=...2_212_{2k}\alpha_k^12_{2k}12_4=...2_212_{2k}12_{2k+1}12^*2_{2k-2}12_{2k}12_4....$$
\end{corollary}

Let $\alpha_k^2=12_{2k}\alpha_k^12_{2k}1=12_{2k}12_{2k+1}12^*2_{2k-2}12_{2k}1$. The word $\theta=...2_2\alpha_k^22_4$ in the conclusion of the previous corollary continues as $\theta=...2_{a}\alpha^2_k2_b...$ with $a\geq2$, $b\geq 4$. If $a>2k-1$ and $b>2k+1$, then $\lambda_0^-(\theta)>m(\gamma_k^1)$. Thus, we have four cases:
\begin{itemize}
\item[Ext2A)] The string $2_{2k-1}{\alpha}_k^22_{2k+1}$.

\item[Ext2B)] The string $\Delta_{a,b}=12_a{\alpha}_k^22_b1$, with $a<2k-1$ and $b<2k+1$.

\item[Ext2C)] The string $\Delta_a=12_{a}{\alpha}_k^22_{2k+1}$, with $a<2k-1$.

\item[Ext2D)] The string $\Delta^b=2_{2k-1}{\alpha}_k^22_{b}1$, with $b<2k+1$.

\end{itemize}
\subsubsection{Ruling out Ext2B)} This case essentially never occurs. In fact, by the Lemma \ref{bpodd}, $a$ can not be odd in this regime. It remains the case where $a=2j<2k-1$ is even. Again by the Lemma \ref{bpodd}, $\lambda^-_0(2_{2k-2}12^*2_{2m}1)>m(\gamma^1_k)$, $m\le k-2$, so that if $b<2k+1$ is odd, then we must have $b=2k-1$. In particular, we are left with the possibilities that $b=2k-1$ or $b<2k+1$ is even. In order to eliminate these cases, we use the next two lemmas:

\begin{lemma} Let $\Delta_{a,b}=12_{a}\alpha_k^22_b1=12_a12_{2k}12_{2k+1}12^*2_{2k-2}12_{2k}12_b1$. We have:
\begin{itemize}
\item[i)] $\lambda_0^+(\Delta_{2k-2,2k-1})<\lambda_0^+(\Delta_{2k-4,2k-1})<m(\theta(\underline{\omega}_k))$;
\item[ii)] $\lambda_0^-(\Delta_{2j,2k-1})\ge\lambda_0^-(\Delta_{2k-6,2k-1})>m(\gamma_k^1)$ for $2j\leq 2k-6$.
\end{itemize}
\end{lemma}
\begin{proof}
 It is easy to see that $\lambda_0^+(\Delta_{2k-2,2k-1})<\lambda_0^+(\Delta_{2k-4,2k-1})$. In order to show that $\lambda_0^+(\Delta_{2k-4,2k-1})<m(\theta(\underline{\omega}_k))$, we write $\lambda^+_0(\Delta_{2k-4,2k-1}):=A+B$, where
 	$$A=[2;2_{2k-2},1,2_{2k},1,2_{2k-1},1,\overline{1,2}] \quad \mbox{and} \quad B:=[0;1,2_{2k+1},1,2_{2k},1,2_{2k-4},1,\overline{2,1}].$$
Since $m(\theta(\underline{\omega}_k))>C+D$ with 
	$$C:=[2;2_{2k-2},1,2_{2k},1,2_{2k+1},1,\overline{2,1}] \quad \mbox{and} \quad D:=[0;1,2_{2k+1},1,2_{2k},1,2_{2k-1},1,\overline{2,1}],$$ 
our task is reduced to prove that $A+B<C+D$. 
	
	Note that 
	$$C-A=\dfrac{[2;\overline{2,1}]-[1;\overline{1,2}]}{q^2(2_{2k-2}12_{2k}12_{2k-1})([2;\overline{2,1}]+\beta)([1;\overline{1,2}]+\beta)},$$
where $\beta=\beta(2_{2k-2}12_{2k}12_{2k-1})=[0;2_{2k-1},1,2_{2k},1,2_{2k-2}]<[0;\overline{2}].$ Moreover,
	$$B-D=\dfrac{[2;2_2,\overline{1,2}]-[1;\overline{2,1}]}{q^2(12_{2k+1}12_{2k}12_{2k-4})([2;2_2,\overline{1,2}]+\tilde{\beta})([1;\overline{2,1}]+\tilde{\beta})},$$
	where $\tilde{\beta}=\beta(12_{2k+1}12_{2k}12_{2k-4})=[0;2_{2k-4},1,2_{2k},1,2_{2k+1},1]>[0;\overline{2}].$
	Then
	$$\dfrac{C-A}{B-D}=\dfrac{q^2(12_{2k+1}12_{2k}12_{2k-4})}{q^2(2_{2k-2}12_{2k}12_{2k-1})}\cdot X\cdot Y,$$
	where
	$$X=\dfrac{[2;\overline{2,1}]-[1;\overline{1,2}]}{[2;2_2,\overline{1,2}]-[1;\overline{2,1}]}=0.6$$
	and
	$$Y=\dfrac{([2;2_2,\overline{1,2}]+\tilde{\beta})([1;\overline{2,1}]+\tilde{\beta})}{([2;\overline{2,1}]+\beta)([1;\overline{1,2}]+\beta)}>\dfrac{([2;2_2,\overline{1,2}]+[0;\overline{2}])([1;\overline{2,1}]+[0;\overline{2}])}{([2;\overline{2,1}]+[0;\overline{2}])([1;\overline{1,2}]+[0;\overline{2}])}>0.84.$$		
	On the other hand, by Euler's rule, 
	\begin{eqnarray*}
	q(12_{2k+1}12_{2k}12_{2k-4})&=&q(12_2)q(2_{2k-1}12_{2k}12_{2k-4})+q(12)q(2_{2k-2}12_{2k}12_{2k-4}) \\ 
	&=& 7q(2_{2k-1}12_{2k}12_{2k-4})+3q(2_{2k-2}12_{2k}12_{2k-4})
	\end{eqnarray*}
	and
	$$q(2_{2k-2}12_{2k}12_{2k-1})=5q(2_{2k-1}12_{2k}12_{2k-4})+2q(2_{2k-1}12_{2k}12_{2k-5}).$$
	Hence, 
	$$\frac{q(12_{2k+1}12_{2k}12_{2k-4})}{q(2_{2k-2}12_{2k}12_{2k-1})} = \frac{7+3\beta(2_{2k-4}12_{2k}12_{2k-1})}{5+2\beta(2_{2k-1}12_{2k}12_{2k-4})}>1.41.$$ 
	 In particular, 
	$$\dfrac{C-A}{B-D}>(1.41)^2\cdot 0.6\cdot 0.84>1.001 >1.$$
	
	To prove ii) we write $\lambda^-_0(\Delta_{2k-6,2k-1})=A'+B'$ with 
	$$B':=[2;2_{2k-2},1,2_{2k},1,2_{2k-1},1,\overline{2,1}] \quad \mbox{and} \quad A':=[0;1,2_{2k+1},1,2_{2k},1,2_{2k-6},1,\overline{1,2}],$$
	and $m(\gamma^1_k)<C'+D'$ with 
	$$D':=[2;2_{2k-2},1,2_{2k},1,2_{2k+1},1,\overline{1,2}] \quad \mbox{and} \quad C':=[0;1,2_{2k+1},1,2_{2k},1,2_{2k-1},1,\overline{1,2}].$$
	Let $\underline{c}=2_{2k-2}12_{2k}12_{2k-1}$ and $\underline{d}=12_{2k+1}12_{2k}12_{2k-6}$. By Lemma \ref{bi} i) and Euler's rule, we have   
	$$q(\underline{d})=q(12_3)q(2_{2k-2}12_{2k}12_{2k-6})+q(12_2)q(2_{2k-3}12_{2k}12_{2k-6}) < \dfrac{41}{2}q(2_{2k-2}12_{2k}12_{2k-6})$$ 
	and 
	$$q(\underline{c})=q(2_4)q(2_{2k-5}12_{2k}12_{2k-2})+q(2_3)q(2_{2k-6}12_{2k}12_{2k-2}) > 70 q(2_{2k-2}12_{2k}12_{2k-6}),$$
	so that $q(\underline{c})>3 \cdot q(\underline{d})$. Since $A'>C'$ and $D'>B'$, it follows from Lemma \ref{cL.U1} that $A'+B'>C'+D'$. 
\end{proof}

\begin{lemma}
Let $\Delta_{a,b}=12_{a}\alpha_k^22_b1=12_a12_{2k}12_{2k+1}12^*2_{2k-2}12_{2k}12_b1$, if $b=2m<2k+1$ and $a=2j<2k-1$, then
$$\lambda_0^-(\Delta_{2j,2m})>m(\gamma_k^1)$$ 
\end{lemma}
 \begin{proof}
 If $a=2j\le 2k-2$ and $b=2m\le 2k$, then $\lambda_0^-(\Delta_{a,b})\ge\lambda_0^-(\Delta_{2k-2,2k})$. Hence, it remains to prove that  $\lambda_0^-(\Delta_{2k-2,2k})>m(\gamma_k^1)$. For this sake, note that:
$$C:=[2;2_{2k-2},1,2_{2k},1,2_{2k},1,\overline{1,2}]>[2;2_{2k-2},1,2_{2k},1,2_{2k+1},1,2,\overline{2,1}]=:A\; {\rm and}$$
$$D:=[0;1,2_{2k+1},1,2_{2k},1,2_{2k-2},1,\overline{1,2}]>[0;1,2_{2k+1},1,2_{2k},1,2_{2k-1},1,2,\overline{2,1}]=:B.$$
Therefore, $\lambda_0^-(\Delta_{2k-2,2k}):=C+D>A+B>m(\gamma_k^1)$. 
\end{proof}

\subsubsection{Ruling out Ext2C)} This case essentially never occurs. Again, if $a<2k-1$, then, by Lemma \ref{bpodd}, $a$ can not be odd. It remains the case where $a=2j<2k-1$ is even, which is eliminated by the next lemma:

\begin{lemma} \label{e3.c.e}
 Let $\Delta_a=12_{a}12_{2k}12_{2k+1}12^*2_{2k-2}12_{2k}12_{2k+1}$ with $k\geq 4$. If $a=2j \le 2k-2$, then $\lambda^-_0(\Delta_{2j})\geq \lambda^-_0(\Delta_{2k-2})>m(\gamma_k^1).$
\end{lemma}
\begin{proof} As usual, let us write 
$$\lambda^-_0(\Delta_{2k-2}):=A+B \ \ \mbox{and} \ \  m(\gamma^1_k)<C+D,$$ where
$A=[2;2_{2k-2},1,2_{2k},1,2_{2k+1},\overline{1,2}]$, $B=[0;1,2_{2k+1},1,2_{2k},1,2_{2k-2},1,\overline{1,2}]$, 
$$C=[2;2_{2k-2},1,2_{2k},1,2_{2k+1},1,2_{2},\overline{1,2}], \ \ \ D:=[0;1,2_{2k+1},1,2_{2k},1,2_{2k-1},1,\overline{1,2}].$$
Then,  
	$$C-A=\dfrac{[2;\overline{1,2}]-[1;\overline{2,1}]}{q^2(\underline{c})([2;\overline{1,2}]+\beta(\underline{c}))([1;\overline{2,1}]+\beta(\underline{c}))}$$
	and 
	$$B-D=\dfrac{[2;1,\overline{1,2}]-[1;\overline{1,2}]}{q^2(\underline{d})([2;1,\overline{1,2}]+\beta(\underline{d}))([1;\overline{1,2}]+\beta(\underline{d}))}.$$
	where $\underline{c}=2_{2k-2}12_{2k}12_{2k+1}12$ and $\underline{d}=12_{2k+1}12_{2k}12_{2k-2}$. It follows that
	$$\dfrac{B-D}{C-A}=\dfrac{q^2(\underline{c})}{q^2(\underline{d})}\cdot \dfrac{[2;1,\overline{1,2}]-[1;\overline{1,2}]}{[2;\overline{1,2}]-[1;\overline{2,1}]}\cdot Y >\dfrac{q^2(\underline{c})}{q^2(\underline{d})}\cdot 0.61\cdot Y,$$
	where 
	$$Y=\dfrac{([2;\overline{1,2}]+\beta(\underline{c}))([1;\overline{2,1}]+\beta(\underline{c}))}{([2;1,\overline{1,2}]+\beta(\underline{d}))([1;\overline{1,2}]+\beta(\underline{d}))}.$$
	Since 
	$$\beta(\underline{c})=[0;2,1,2_{2k+1},1,2_{2k},1,2_{2k-2}]>[0;2,1,2_9]>0.369$$
	 and 
	 $$\beta(\underline{d})=[0;2_{2k-2},1,2_{2k},1,2_{2k+1},1]< [0;2_6,1]$$ 
	 we have
	$$Y>\dfrac{([2;\overline{1,2}]+0.369)([1;\overline{2,1}]+0.369)}{([2;1,\overline{1,2}]+[0;2_6,1])([1;\overline{1,2}]+[0;2_6,1])}>0.83.$$
	Because $q(\underline{c})>2q(\underline{d})$, we conclude that 
	$$\dfrac{B-D}{C-A}>2^2\cdot 0.61\cdot 0.83>2,$$ 
	i.e., $A+B>C+D$. 
\end{proof}

 \subsubsection{Ruling out Ext2D)}
 This case essentially never occurs. Indeed, if $b=2m+1<2k+1$ is odd, then Lemma \ref{bpodd} forces $b=2k-1$. This subcase is eliminated by the next lemma:
 \begin{lemma} Let $\Delta^b=2_{2k-1}12_{2k}12_{2k+1}12^*2_{2k-2}12_{2k}12_{b}1$. We have $\lambda_0^+(\Delta^{2k-1})<m(\theta(\underline{\omega}_k))$.
 \end{lemma}
 \begin{proof}
 By definition, $m(\theta(\underline{\omega}_k))>A+B$, where $A:=[2;2_{2k-2},1,2_{2k},1,2_{2k+1},1,\overline{2,1}]$ and $B:=[0;1,2_{2k+1},1,2_{2k},1,2_{2k-1},1,2,2,\overline{2,1}]$. Note that $\lambda_0^+(\Delta^{2k-1})=C+D$, where $C:=[2;2_{2k-2},1,2_{2k},1,2_{2k-1},1,\overline{1,2}]$ and $D:=[0;1,2_{2k+1},1,2_{2k},1,2_{2k-1},\overline{2,1}]$. Hence, our work is reduced to prove that $A-C>D-B$. 
 
 In order to prove this inequality, note that $A>C$, $D>B$, and, by Euler's rule, 
 $$q(12_{2k+1}12_{2k}12_{2k-1})>q(2_{2k-1}12_{2k}12_{2k-2})q(2_31) = 17 q(2_{2k-2}12_{2k}12_{2k-1}).$$ 
 Therefore, the desired inequality follows from Lemma \ref{cL.U1}. 
 \end{proof}
 
It remains the subcase where $b=2m<2k+1$ is even, but this possibility does not occur thanks to the next lemma: 
\begin{lemma}\label{L.U3.17} Let $\Delta^b=2_{2k-1}12_{2k}12_{2k+1}12^*2_{2k-2}12_{2k}12_{b}1$. 
If $b=2m<2k+1$, then
$\lambda^-_0(\Delta^{2m})\geq \lambda^-_0(\Delta^{2k})>m(\gamma_k^1)$.
\end{lemma}
\begin{proof}
It is not hard to show that $\lambda^-_0(\Delta^{2m})\geq \lambda^-_0(\Delta^{2k})$ for $2m\leq 2k$. To see that $\lambda^-_0(\Delta^{2k})>m(\gamma_k^1)$, we write 
	$$\lambda^-_0(\Delta^{2k})=[2;2_{2k-2},1,2_{2k},1,2_{2k},1,\overline{1,2}]+[0;1,2_{2k+1},1,2_{2k},1,2_{2k-1},\overline{1,2}]:=A+B$$
and
	$$m(\gamma^1_k)<[2;2_{2k-2},1,2_{2k},1,2_{2k+1},1,\overline{1,2}]+[0;1,2_{2k+1},1,2_{2k},1,2_{2k-1},1,\overline{1,2}]:=C+D.$$
Note that $A>C$ and $D>B$. Moreover,
	$$q(12_{2k+1}12_{2k}12_{2k-1}1)>q(12_3)q(2_{2k-2}12_{2k}12_{2k-1}1) = 17 q(2_{2k-2}12_{2k}12_{2k-1}1)$$
and 
$$q(2_{2k-2}12_{2k}12_{2k})<3q(2_{2k-2}12_{2k}12_{2k-1}).$$
	In particular, $q(12_{2k+1}12_{2k}12_{2k-1}1)>4q(2_{2k-2}12_{2k}12_{2k})$ and, by Lemma \ref{cL.U1}, we have $A+B>C+D.$
\end{proof}

\subsubsection{Conclusion: Ext2B), Ext2C), Ext2D) are ruled out}

Our discussion after Corollary \ref{c.Ext2-0} until now implies that Ext2A) is essentially the sole possible extension of $\theta=2_2\alpha_k^22_4$: in fact, we have proved that 
\begin{corollary}\label{c.Ext2-1}
There exists an explicit parameter $\lambda_k^{(5)}>m(\gamma_k^1)$ such that any $(k,\lambda_k^{(5)})$-admissible word $\theta$ containing $2_2\alpha_k^22_4$ extends as $$\theta=...2_{2k-1}\alpha_k^22_{2k+1}=...2_{2k-1}12_{2k}12_{2k+1}12^*2_{2k-2}12_{2k}12_{2k+1}....$$
\end{corollary}

\subsection{Extension from $2_{2k-1}\alpha_k^22_{2k+1}$ to $2_{2k+1}12_{2k-1}\alpha_k^22_{2k+1}12_{2k-1}$}

\begin{lemma} Let $\alpha_k^2=12_{2k}12_{2k+1}12^*2_{2k-2}12_{2k}1$. We have:
\begin{itemize} \label{t2-3}
\item[i)] $\lambda^-_0(2_{2k-1}\alpha_k^22_{2k+1}2)>\lambda^-_0(2_{2k-1}\alpha_k^22_{2k+1}11)>m(\gamma_k^1);$
\item[ii)] $\lambda^-_0(22_{2k-1}\alpha_k^22_{2k+1}122)>\lambda^-_0(112_{2k-1}\alpha_k^22_{2k+1}122)>m(\gamma_k^1);$
\end{itemize}
\end{lemma}
\begin{proof}
The inequality $\lambda^-_0(2_{2k-1}\alpha_k^22_{2k+1}2)>\lambda^-_0(2_{2k-1}\alpha_k^22_{2k+1}11)$ is straightforward. Thus, the proof of item i) is reduced to check that $\lambda^-_0(2_{2k-1}\alpha_k^22_{2k+1}11)>m(\gamma_k^1)$. In order to do this, we write  $m(\gamma_k^1)<[2;2_{2k-2},1,2_{2k},1,2_{2k+1},1,2,2,\overline{1,2}]+[0;1,2_{2k+1},1,2_{2k},1,2_{2k-1},1,2_4,\overline{1,2}]:=A+B$. Note that $\lambda^-_0(2_{2k-1}\alpha_k^22_{2k+1}11)=C+D:=[2;2_{2k-2},1,2_{2k},1,2_{2k+1},1,1,\overline{1,2}]+[0;1,2_{2k+1},1,2_{2k},1,2_{2k-1},\overline{1,2}]$. 
 
 Hence, our work is reduced to prove that $C-A>B-D$.
 In order to show this inequality, we observe that:

$$C-A=\dfrac{[2;2,\overline{1,2}]-[1;\overline{1,2}]}{{q}_{6k+2}^2([2;2,\overline{1,2}]+{\beta})([1;\overline{1,2}]+{\beta})}$$
and
$$B-D=\dfrac{[1;2_4,\overline{1,2}]-[1;\overline{2,1}]}{\tilde{q}_{6k+3}^2([1;2_4,\overline{1,2}]+\tilde{\beta})([1;\overline{2,1}]+\tilde{\beta})},$$
where $q_{6k+2}=q(2_{2k-2}12_{2k}12_{2k+1}1)$, $\tilde{q}_{6k+3}=q(12_{2k+1}12_{2k}12_{2k-1})$, $\beta=[0;1,2_{2k+1},1,2_{2k},1,2_{2k-2}]$ and $\tilde{\beta}=[0;2_{2k-1},1,2_{2k},1,2_{2k+1},1]$.
Thus,
$$\dfrac{C-A}{B-D}=\dfrac{[2;2,\overline{1,2}]-[1;\overline{1,2}]}{[1;2_4,\overline{1,2}]-[1;\overline{2,1}]}\cdot Y \cdot \dfrac{\tilde{q}^2_{6k+3}}{{q}^2_{6k+1}} > 13.08\cdot Y \cdot \dfrac{\tilde{q}^2_{6k+3}}{{q}^2_{6k+1}},$$
where
$$Y=\dfrac{([1;2_4,\overline{1,2}]+\tilde{\beta})([1;\overline{2,1}]+\tilde{\beta})}{([2;2,\overline{1,2}]+{\beta})([1;\overline{1,2}]+{\beta})}>\dfrac{([1;2_4,\overline{1,2}]+[0;2_2])([1;\overline{2,1}]+[0;2_2])}{([2;2,\overline{1,2}]+[0;1,2_2])([1;\overline{1,2}]+[0;1,2_2])}>0.42.$$
By Euler's rule and Lemma \ref{bi} i), we have:
\begin{align*}
\tilde{q}_{6k+3}=2q(12_{2k+1}12_{2k}12_{2k-2})+q(12_{2k+1}12_{2k}12_{2k-3})>(2+1/3)q_{6k+2}.
\end{align*} 
Therefore, 
$$\dfrac{C-A}{B-D}> 13.08 \cdot 0.42 \cdot \left( \dfrac{7}{3} \right)^2>29.9>1.$$

Now, we prove ii). By parity, we can easily check that $\lambda^-_0(22_{2k-1}\alpha_k^22_{2k+1}122)>\lambda^-_0(112_{2k-1}\alpha_k^22_{2k+1}122)$. It remains to prove that $\lambda^-_0(112_{2k-1}\alpha_k^22_{2k+1}122)>m(\gamma_k^1)$.
By definition, we have  $m(\gamma_k^1)<A'+B'$ with $A':=[2;2_{2k-2},1,2_{2k},1,2_{2k+1},1,2_4,\overline{1,2}]$ and $B':=[0;1,2_{2k+1},1,2_{2k},1,2_{2k-1},1,2_2,\overline{1,2}]$. Note that $\lambda^-_0(112_{2k-1}\alpha_k^22_{2k+1}122)=[2;2_{2k-2},1,2_{2k},1,2_{2k+1},1,2_2,\overline{2,1}]+[0;1,2_{2k+1},1,2_{2k},1,2_{2k-1},1_2,\overline{1,2}]:=C'+D'$. 
Our task is reduced to show that $D'-B'>A'-C'$. We have:
$$D'-B'=\dfrac{[1;1,\overline{1,2}]-[1;2,2,\overline{1,2}]}{\tilde{q}_{6k+3}^2([1;1,\overline{1,2}]+\tilde{\beta})([1;2,2,\overline{1,2}]+\tilde{\beta})}$$
and
$$A'-C'=\dfrac{[2;2,\overline{2,1}]-[2;2_3,\overline{1,2}]}{{q}_{6k+2}^2([2;2,\overline{2,1}]+{\beta})([2;2_3,\overline{1,2}]+{\beta})},$$
where $q_{6k+2}=q(2_{2k-2}12_{2k}12_{2k+1}1)$, $\tilde{q}_{6k+3}=q(12_{2k+1}12_{2k}12_{2k-1})$, $\beta=[0;1,2_{2k+1},1,2_{2k},1,2_{2k-2}]$ and $\tilde{\beta}=[0;2_{2k-1},1,2_{2k},1,2_{2k+1},1]$.
Thus,
$$\dfrac{D'-B'}{A'-C'}=\dfrac{[1;1,\overline{1,2}]-[1;2,2,\overline{1,2}]}{[2;2,\overline{2,1}]-[2;2_3,\overline{1,2}]}\cdot Y' \cdot \dfrac{{q}^2_{6k+2}}{\tilde{q}^2_{6k+3}} > 15.66\cdot Y' \cdot \dfrac{{q}^2_{6k+2}}{\tilde{q}^2_{6k+3}},$$
where
$$Y'=\dfrac{([2;2,\overline{2,1}]+{\beta})([2;2_3,\overline{1,2}]+{\beta})}{([1;1,\overline{1,2}]+\tilde{\beta})([1;2,2,\overline{1,2}]+\tilde{\beta})}>\dfrac{([2;2,\overline{2,1}]+[0;1\bar{2}])([2;2_3,\overline{1,2}]+[0;1\bar{2}])}{([1;1,\overline{1,2}]+[0;\bar{2}])([1;2,2,\overline{1,2}]+[0;\bar{2}])}>2.66.$$
By Euler's rule and Lemma \ref{bi} i), we have:
\begin{align*}
\tilde{q}_{6k+3}=2q(12_{2k+1}12_{2k}12_{2k-2})+q(12_{2k+1}12_{2k}12_{2k-3})<(2+1/2)q_{6k+2}.
\end{align*} 
Therefore, 
$$\dfrac{D'-B'}{A'-C'}> 15.66 \cdot 2.66 \cdot \left( \dfrac{2}{5} \right)^2>6.65>1.$$
\end{proof}

\begin{corollary}\label{c.Ext3-0}
Consider the parameter
$$\lambda_k^{(6)}:=\min\{\lambda_0^-(12^*1),\lambda^-_0(2_{2k-1}\alpha_k^22_{2k+1}11),\lambda^-_0(112_{2k-1}\alpha_k^22_{2k+1}122)\}.$$
Then, $\lambda_k^{(6)}>m(\gamma_k^1)$ and any $(k,\lambda_k^{(6)})$-admissible word $\theta$ containing $2_{2k-1}\alpha_k^22_{2k+1}$ extends as $$\theta=...2_212_{2k-1}\alpha_k^22_{2k+1}12_2=...2_212_{2k-1}12_{2k}12_{2k+1}12^*2_{2k-2}12_{2k}12_{2k+1}12_2....$$
\end{corollary}

Denote $\alpha_k^3=12_{2k-1}\alpha_k^22_{2k+1}1=12_{2k-1}12_{2k}12_{2k+1}12^*2_{2k-2}12_{2k}12_{2k+1}1$. We continue the word $\theta=...2_2\alpha_k^3 2_2$ as $\theta=...2_{a} \alpha_k^3 2_b...$. If $a>2k+1$ and $b>2k-1$, then $\lambda_0^-(\theta)>m(\gamma_k^1)$. Thus, we have four cases:
\begin{itemize}
\item[Ext3A)] The string $2_{2k+1}{\alpha}_k^32_{2k-1}$.

\item[Ext3B)] The string $\Omega_{a,b}=12_a{\alpha}_k^32_b1$, with $a<2k+1$ and $b<2k-1$.

\item[Ext3C)] The string $\Omega_a=12_{a}{\alpha}_k^32_{2k-1}$, with $a<2k+1$.

\item[Ext3D)] The string $\Omega^b=2_{2k+1}{\alpha}_k^32_{b}1$, with $b<2k-1$.

\end{itemize}
\subsubsection{Ruling out Ext3B)} This case essentially never occurs. In fact, if $b=2m+1<2k-1$ is odd, then Lemma \ref{bpodd} says that this string contains a $k$-prohibited string. Thus, it remains $b=2m<2k-1$ even. Analogously, the case $a$ is odd with $a=2j+1<2k-1$ is also eliminate by Lemma \ref{bpodd}. In the case $a=2k-1$, we use the Lemma \ref{ooe} i)  to show that the word $\Omega_{2k-1,b}$ contains a $k$-prohibited string. Thus, it remain just the case where both $a$ and $b$ are even. As it turns out, this case is eliminated by the next lemma:
\begin{lemma} Let $\Omega_{a,b}=12_a12_{2k-1}12_{2k}12_{2k+1}12^*2_{2k-2}12_{2k}12_{2k+1}12_b1$. If $a=2j\le 2k$ and $b=2m \le 2k-2$, then
 $\lambda^-_0(\Omega_{2j,2m})>m(\gamma_k^1).$
\end{lemma}
\begin{proof}
This follows from the fact that 
	$$[0;2_{2k-2},1,2_{2k},1,2_{2k+1},1,2_{2m},1,...]>[0;2_{2k-2},1,2_{2k},1,2_{2k+1},1,2_{2k-1},1,...]$$
and
	$$[0;1,2_{2k+1},1,2_{2k},1,2_{2k-1},1,2_{2j},1,...]>[0;1,2_{2k+1},1,2_{2k},1,2_{2k-1},1,2_{2k+1},1,...]$$ 
	whenever $j\leq k$ and $m\leq k-1$. 
\end{proof}
\subsubsection{Ruling out Ext3C)} This case essentially never occurs. Indeed, by Lemma \ref{bpodd}, $a$ can not be of the form $a=2j+1<2k-1$. Moreover, the case $a=2k-1$ is not possible by Lemma \ref{ooe} i). It remains the case $a=2j<2k+1$, which is eliminated by the following lemma (together with Lemma \ref{p1} i)): 

\begin{lemma} \label{e4.c.e}
 Let $\Omega_a=12_a12_{2k-1}12_{2k}12_{2k+1}12^*2_{2k-2}12_{2k}12_{2k+1}12_{2k-1}$. If $a=2j<2k+1$, then
 $\lambda^-_0(\Omega_{2j}122)\ge\lambda^-_0(\Omega_{2k}122)>m(\gamma_k^1).$ Moreover, for every $2j<2k+1$, one has $\lambda^-_0(\Omega_{2j}2)>\lambda^-_0(\Omega_{2j}11)>\lambda^-_0(\Omega_{2j}122)$.
\end{lemma}
\begin{proof}
By parity, the inequalities $\lambda^-_0(\Omega_{2j}2)>\lambda^-_0(\Omega_{2j}11)>\lambda^-_0(\Omega_{2j}12_2)\ge\lambda^-_0(\Omega_{2k}12_2)$ for $2j\le 2k$ are clear. Now, we show that $\lambda^-_0(\Omega_{2k}122)>m(\gamma_k^1)$. In order to do this, we write $\lambda^-_0(\Omega_{2k}122)=[2;2_{2k-2},1,2_{2k},1,2_{2k+1},1,2_{2k-1},1,2_2,\overline{2,1}]+[0;1,2_{2k+1},1,2_{2k},1,2_{2k-1},1,2_{2k},1,\overline{1,2}]:=C+D$ and  $m(\gamma_k^1)<A+B$, where $$A:=[2;2_{2k-2},1,2_{2k},1,2_{2k+1},1,2_{2k-1},1,2_6,\overline{1,2}] \quad {\rm \; and}$$ $$B:=[0;1,2_{2k+1},1,2_{2k},1,2_{2k-1},1,2_{2k+1},1,2_2,\overline{1,2}].$$ 
In this context, our task is reduced to prove that $D-B>A-C$. We observe that:
$$D-B=\dfrac{[2;1,\overline{1,2}]-[2;2,1,2_2,\overline{1,2}]}{\tilde{q}_{8k+3}^2([2;1,\overline{1,2}]+\tilde{\beta})([2;2,1,2_2,\overline{1,2}]+\tilde{\beta})}$$
and
$$A-C=\dfrac{[2;2,\overline{2,1}]-[2;2_5,\overline{1,2}]}{{q}_{8k+2}^2([2;2,\overline{2,1}]+{\beta})([2;2_5,\overline{1,2}]+{\beta})},$$
where $q_{8k+2}=q(2_{2k-2}12_{2k}12_{2k+1}12_{2k-1}1)$, $\tilde{q}_{8k+3}=q(12_{2k+1}12_{2k}12_{2k-1}12_{2k-1})$, $\beta=[0;1,2_{2k-1},1,2_{2k+1},1,2_{2k},1,2_{2k-2}]$ and $\tilde{\beta}=[0;2_{2k-1},1,2_{2k-1},1,2_{2k},1,2_{2k+1},1]$.
Thus,
$$\dfrac{D-B}{A-C}=\dfrac{[2;1,\overline{1,2}]-[2;2,1,2_2,\overline{1,2}]}{[2;2,\overline{2,1}]-[2;2_5,\overline{1,2}]}\cdot Y \cdot \dfrac{{q}^2_{8k+2}}{\tilde{q}^2_{8k+3}} > 24.45\cdot Y \cdot \dfrac{{q}^2_{8k+2}}{\tilde{q}^2_{8k+3}},$$
where
$$Y=\dfrac{([2;2,\overline{2,1}]+{\beta})([2;2_5,\overline{1,2}]+{\beta})}{([2;1,\overline{1,2}]+\tilde{\beta})([2;2,1,2_2,\overline{1,2}]+\tilde{\beta})}>\dfrac{([2;2,\overline{2,1}]+[0;1,\bar{2}])([2;2_5,\overline{1,2}]+[0;1,\bar{2}])}{([2;1,\overline{1,2}]+[0;\bar{2}])([2;2,1,2_2,\overline{1,2}]+[0;\bar{2}])}>1.17.$$
Let $\Gamma=2_{2k-2}12_{2k}12_{2k+1}1$ and $\Sigma=2_{2k-1}1$. By Euler's rule and Lemma \ref{bi} i), we have:
\begin{align*}
{q}_{8k+2}=q(\Gamma)q(\Sigma)+q(2_{2k-2}12_{2k}12_{2k+1})q(2_{2k-2}1)>q(\Gamma)q(\Sigma)(1+2/3\cdot1/3), 
\end{align*} 
\begin{align*}
\tilde{q}_{8k+3}=q(12_{2k+1}12_{2k}12_{2k-1})q(\Sigma^t)+q(\Gamma^t)q(2_{2k-1})<q(\Gamma^t)q(\Sigma^t)(3+3/4).
\end{align*} 

Thus, 
$$\dfrac{D-B}{A-C}> 24.45 \cdot 1.17 \cdot \left( \dfrac{44}{135} \right)^2>3>1.$$
\end{proof}

\subsubsection{Ruling out Ext3D)} This case essentially never occurs. Indeed, by Lemma \ref{bpodd}, $b$ can not be of the form $b=2m+1<2k-1$. Thus, it remains the case $b=2m<2k-1$ even. As it turns out, this case is excluded by the following lemma: 

\begin{lemma}\label{L.U3.21}  Let $\Omega^b=2_{2k+1}12_{2k-1}12_{2k}12_{2k+1}12^*2_{2k-2}12_{2k}12_{2k+1}12_{b}1$. If $b=2m<2k-1$, then
 $\lambda^-_0(\Omega^{2m})\ge\lambda^-_0(\Omega^{2k-2})>m(\gamma_k^1).$
\end{lemma}
\begin{proof}
It follows the same ideia of Lemma \ref{L.U3.17}. In fact, let $\underline{c}=2_{2k-2}12_{2k}12_{2k+1}12_{2k-2}$ and $\underline{d}=12_{2k+1}12_{2k}12_{2k-1}12_{2k+1}$, and denote 
	$$A=[2;\underline{c},1,\overline{1,2}] \ \ \mbox{and} \ \ B=[0;\underline{d},\overline{1,2}]$$
and 
	$$C=[2;\underline{c},2,\overline{2,1}] \ \ \mbox{and} \ \ D=[0;\underline{d},\overline{2,1}].$$
One can check that $\lambda^-_0(\Omega^{2k-2})=A+B$, $m(\gamma^1_k)<C+D$, $A>C$ and $D>B$. Also, Euler's rule implies $q(\underline{d})>4q(\underline{c})$, so that $A+B>C+D$ thanks to Lemma \ref{cL.U1}.
\end{proof}

\subsubsection{Conclusion: Ext3B), Ext3C) and Ext3D) are ruled out} Our discussion after Corollary \ref{c.Ext3-0} until now implies that Ext3A) is essentially the sole possible extension of $\theta=2_2\alpha_k^32_2$: in fact, we have proved that 

\begin{corollary}\label{c.Ext3-1}
There exists an explicit parameter $\lambda_k^{(7)}>m(\gamma_k^1)$ and any $(k,\lambda_k^{(7)})$-admissible word $\theta$ containing $2_2\alpha_k^32_2$ extends as $$\theta=...2_{2k+1}\alpha_k^32_{2k-1}=...2_{2k+1}12_{2k-1}12_{2k}12_{2k+1}12^*2_{2k-2}12_{2k}12_{2k+1}12_{2k-1}....$$
\end{corollary}

\subsection{End of proof of Theorem \ref{t.extension}}

From Corollaries \ref{c.Ext1-0}, \ref{c.Ext1-1}, \ref{c.Ext2-0}, \ref{c.Ext2-1}, \ref{c.Ext3-0}, \ref{c.Ext3-1}, we see that the statement of Theorem \ref{t.extension} is true for $\mu_k^{(1)}:=\min\{\lambda_k^{(i)}:i=2,\dots, 7\}$. 

\section{Replication mechanism for $\gamma_k^1$}\label{s.replication}

In this section, we investigate the extension of a word $\theta$ containing the string
\begin{equation*}
\alpha_k^4:=2_{2k+1}12_{2k-1}12_{2k}12_{2k+1}12^*2_{2k-2}12_{2k}12_{2k+1}12_{2k-1}
\end{equation*}
\begin{lemma} \label{srl}
 We have:
\begin{itemize}
\item[i)] $\lambda^-_0(\alpha_k^42)>\lambda^-_0(\alpha_k^411)>\lambda^-_0(\alpha_k^41221)>m(\gamma_k^1);$
\item[ii)] $\lambda^-_0(2\alpha_k^412_4)>\lambda^-_0(11\alpha_k^412_4)>m(\gamma_k^1).$
\end{itemize}
\end{lemma}
\begin{proof}
By parity, we get the inequalities  $\lambda^-_0(\alpha_k^42)>\lambda^-_0(\alpha_k^411)>\lambda^-_0(\alpha_k^41221)$. Thus, the proof of i) is reduced to check the inequality $\lambda^-_0(\alpha_k^41221)>m(\gamma_k^1)$. In this direction,  we write $m(\gamma_k^1)<[2;2_{2k-2},1,2_{2k},1,2_{2k+1},1,2_{2k-1},1,2_4,\overline{1,2}]+[0;1,2_{2k+1},1,2_{2k},1,2_{2k-1},1,2_{2k+1},1,2_4,\overline{1,2}]:=A+B$ and we note that $\lambda^-_0(\alpha_k^41221)=C+D$, where $$C:=[2;2_{2k-2},1,2_{2k},1,2_{2k+1},1,2_{2k-1},1,2,2,1,\overline{1,2}]$$ and $$D:=[0;1,2_{2k+1},1,2_{2k},1,2_{2k-1},1,2_{2k+1},\overline{1,2}].$$ 
 Hence, our work is reduced to prove that $C-A>B-D$. In order to prove this estimate, we observe that:
$$C-A=\dfrac{[1;2_2,1,\overline{1,2}]-[1;2_4,\overline{1,2}]}{{q}_{8k+1}^2([1;2_2,1,\overline{1,2}]+{\beta})([1;2_4,\overline{1,2}]+{\beta})}$$
and
$$B-D=\dfrac{[2;2,1,2_4,\overline{1,2}]-[2;2,1,2,\overline{1,2}]}{\tilde{q}_{8k+3}^2([2;2,1,2_4,\overline{1,2}]+\tilde{\beta})([2;2,1,2,\overline{1,2}]+\tilde{\beta})},$$
where $q_{8k+1}=q(2_{2k-2}12_{2k}12_{2k+1}12_{2k-1})$, $\tilde{q}_{8k+3}=q(12_{2k+1}12_{2k}12_{2k-1}12_{2k-1})$, $\beta=[0;2_{2k-1},1,2_{2k+1},1,2_{2k},1,2_{2k-2}]$ and $\tilde{\beta}=[0;2_{2k-1},1,2_{2k-1},1,2_{2k},1,2_{2k+1},1]$.
Thus,
$$\dfrac{C-A}{B-D}=\dfrac{[1;2_2,1,\overline{1,2}]-[1;2_4,\overline{1,2}]}{[2;2,1,2_4,\overline{1,2}]-[2;2,1,2,\overline{1,2}]}\cdot Y \cdot \dfrac{\tilde{q}^2_{8k+3}}{{q}^2_{8k+1}} > 1.26\cdot Y \cdot \dfrac{\tilde{q}^2_{8k+3}}{{q}^2_{8k+1}},$$
where
$$Y=\dfrac{([2;212_4\overline{12}]+\tilde{\beta})([2;212\overline{12}]+\tilde{\beta})}{([1;2_21\overline{12}]+{\beta})([1;2_4\overline{12}]+{\beta})}> \dfrac{([2;212_4\overline{12}]+[0;2_4])([2;212\overline{12}]+[0;2_4])}{([1;2_21\overline{12}]+[0;\overline{2}])([1;2_4\overline{12}]+[0;\overline{2}])}>2.3.$$

Let $\Gamma=12_{2k+1}12_{2k}12_{2k-2}$ and $\Sigma=2_{2k-1}$. By Euler's rule and Lemma \ref{bi} i):
\begin{align*}
\tilde{q}_{8k+3}&=q(12_{2k+1}12_{2k}12_{2k-1})q(12_{2k-1})+q(\Gamma)q(\Sigma)>\frac{4}{3}q(12_{2k+1}12_{2k}12_{2k-1})q(\Sigma)+q(\Gamma)q(\Sigma)\\
&=\frac{4}{3}q(\Sigma)\left[2q(12_{2k+1}12_{2k}12_{2k-2})+q(12_{2k+1}12_{2k}12_{2k-3})\right]+q(\Gamma)q(\Sigma)\\
&>q(\Gamma)q(\Sigma)\left[ 4/3(2+1/3)+1\right]=37q(\Gamma)q(\Sigma)/9
\end{align*} 
and 
\begin{align*}
{q}_{8k+1}&=q(\Gamma^T)q(\Sigma)+q(2_{2k-2}12_{2k}12_{2k+1})q(2_{2k-2})<q(\Gamma)q(\Sigma)\left(1+\frac{3}{4}\cdot\frac{1}{2}\right)=\frac{11}{8}q(\Gamma)q(\Sigma).
\end{align*} 
Therefore, 
$$\dfrac{C-A}{B-D}> 1.26 \cdot 2.3 \cdot \left( \dfrac{296}{99} \right)^2>1.$$

Now, we prove ii). By parity, we can easily check that $\lambda^-_0(2\alpha_k^412_4)>\lambda^-_0(11\alpha_k^412_4)$. It remains to prove that $\lambda^-_0(11\alpha_k^412_4)>m(\gamma_k^1).$
We have  $m(\gamma_k^1)<A'+B':=[2;2_{2k-2},1,2_{2k},1,2_{2k+1},1,2_{2k-1},1,2_8,\overline{1,2}]+[0;1,2_{2k+1},1,2_{2k},1,2_{2k-1},1,2_{2k+1},1,2_4,\overline{1,2}]$. Also, $\lambda^-_0(11\alpha_k^412_4)=C'+D'$ with $C':=[2;2_{2k-2},1,2_{2k},1,2_{2k+1},1,2_{2k-1},1,2_4,\overline{2,1}]$ and $D':=[0;1,2_{2k+1},1,2_{2k},1,2_{2k-1},1,2_{2k+1},1_2,\overline{1,2}]$. 
Hence, our task is reduced to show that $D'-B'>A'-C'$. We have:
$$D'-B'=\dfrac{[2;2,1_2,\overline{1,2}]-[2;2,1,2_4,\overline{1,2}]}{\tilde{q}_{8k+3}^2([2;2,1_2,\overline{1,2}]+\tilde{\beta})([2;2,1,2_4,\overline{1,2}]+\tilde{\beta})}$$
and
$$A'-C'=\dfrac{[2;2_3,\overline{2,1}]-[2;2_7,\overline{1,2}]}{{q}_{8k+2}^2([2;2_3,\overline{2,1}]+{\beta'})([2;2_7,\overline{1,2}]+{\beta'})},$$
where $q_{8k+2}=q(2_{2k-2}12_{2k}12_{2k+1}12_{2k-1}1)$, $\tilde{q}_{8k+3}=q(12_{2k+1}12_{2k}12_{2k-1}12_{2k-1})$, $\beta'=[0;1,2_{2k-1},1,2_{2k+1},1,2_{2k},1,2_{2k-2}]$ and $\tilde{\beta}=[0;2_{2k-1},1,2_{2k-1},1,2_{2k},1,2_{2k+1},1]$.
Thus,
$$\dfrac{D'-B'}{A'-C'}=\dfrac{[2;2,1_2,\overline{1,2}]-[2;2,1,2_4,\overline{1,2}]}{[2;2_3,\overline{2,1}]-[2;2_7,\overline{1,2}]}\cdot Y \cdot \dfrac{{q}^2_{8k+2}}{\tilde{q}^2_{8k+3}} >41.14\cdot Y' \cdot \dfrac{{q}^2_{8k+2}}{\tilde{q}^2_{8k+3}},$$
where
$$Y'=\dfrac{([2;2_3,\overline{2,1}]+{\beta'})([2;2_7,\overline{1,2}]+{\beta'})}{([2;2,1_2,\overline{1,2}]+\tilde{\beta})([2;2,1,2_4,\overline{1,2}]+\tilde{\beta})}>\frac{([2;2_3\overline{21}]+[0;\overline{2}])([2;2_7\overline{12}]+[0;\overline{2}])}{([2;21_2\overline{12}]+[0;\overline{2}])([2;212_4\overline{12}]+[0;\overline{2}])}>1.$$

Let $\tilde{\Gamma}=2_{2k-2}12_{2k}12_{2k+1}1$ and $\tilde{\Sigma}=2_{2k-1}1$. By Euler's rule and Lemma \ref{bi} ii):
\begin{align*}
{q}_{8k+2}=q(\tilde{\Gamma})q(\tilde{\Sigma})+q(2_{2k-2}12_{2k}12_{2k+1})q(2_{2k-2}1)>q(\tilde{\Gamma})q(\tilde{\Sigma})(1+(12/17)\cdot (7/17)),
\end{align*} 
\begin{align*}
\tilde{q}_{8k+3}=q(\tilde{\Gamma}^T2)q(\tilde{\Sigma})+q(\tilde{\Gamma}^T)q(2_{2k-1})<q(\tilde{\Gamma})q(\tilde{\Sigma})(17/7+17/24).
\end{align*} 
Therefore,
$$\dfrac{D-B}{A-C}> 41.14 \cdot \left(\dfrac{373\cdot 168}{289 \cdot 527}\right)^2>6.96>1.$$
\end{proof}

A direct consequence of the previous lemma and Lemmas \ref{bpodd} and \ref{p1} i) is: 

\begin{corollary} \label{crl1}
Consider the parameter
$$\lambda_k^{(8)}:=\min\{\lambda_0^-(12^*1),\lambda^-_0(2_{2k-2}12^*2_21),\lambda^-_0(\alpha_k^41221),\lambda^-_0(11\alpha_k^412_4)\}.$$
Then, $\lambda_k^{(8)}>m(\gamma_k^1)$ and the neighbourhood of the string $\alpha_k^4$ in any $(k,\lambda_k^{(8)})$-admissible word $\theta$ has the form $$\theta=...2_21\alpha_k^412_{4}=...2212_{2k+1}12_{2k-1}12_{2k}12_{2k+1}12^*2_{2k-2}12_{2k}12_{2k+1}12_{2k-1}12_4....$$
\end{corollary}

\subsection{Extension from $2_{2}1\alpha_k^412_{4}$ to $2212_{2k}1\alpha_k^412_{2k}12_{4}$}

Let $\theta=...2_{2}1\alpha_k^412_{4}...$. It extends as $\theta=...2_{a}1\alpha_k^412_{b}...$ with $a\geq 2$, $b\geq 4$. By Lemma \ref{te2-3.1} and Lemma \ref{te2-3.2}, respectively  we have that $b\le 2k$ and $a\le 2k$. Using Lemma \ref{e2.c.o}, we get that $b$ can not be odd. Using Lemmas \ref{bpodd} and \ref{e2.d.o}, we have that $a$ can not be odd. Thus, it remains the cases where $a=2j$ and $b=2m$ are both even. We have four cases:
\begin{itemize}
\item[Rep1)] $a=2k$ and $b=2k$;
\item[Rep2)] $a=2j<2k$ and $b=2m<2k$;
\item[Rep3)] $a=2k$ and $b=2m<2k$;
\item[Rep4)]$a=2j<2k$ and $b=2k$;
\end{itemize}

The case Rep2) essentially never occurs by the next lemma:
\begin{lemma}
If $a=2j<2k$ and $b=2m<2k$, then $\lambda^-_0(12_{2j}1\alpha_k^412_{2m}1)>m(\gamma_k^1)$.
\end{lemma}
\begin{proof}
 For $a=2j\le 2k-2$ and $b=2m\le 2k-2$, the inequality $\lambda_0^-(12_{2j}1\alpha_k^412_{2m}1)\ge\lambda_0^-(12_{2k-2}1\alpha_k^412_{2k-2}1)$ is straightforward. Hence, it remains to prove that  $$\lambda_0^-(12_{2k-2}1\alpha_k^412_{2k-2}1)>m(\gamma_k^1).$$
  For this sake, note that $C>A$ and $D>B$, where:
\begin{align*}
C&:=[2;2_{2k-2},1,2_{2k},1,2_{2k+1},1,2_{2k-1},1,2_{2k-2},1,\overline{1,2}],\\
A&:=[2;2_{2k-2},1,2_{2k},1,2_{2k+1},1,2_{2k-1},1,2_{2k},1,\overline{2,1}],\\
D&:=[0;1,2_{2k+1},1,2_{2k},1,2_{2k-1},1,2_{2k+1},1,2_{2k-2},1,\overline{1,2}] {\rm \; and}\\
B&:=[0;1,2_{2k+1},1,2_{2k},1,2_{2k-1},1,2_{2k+1},1,2_{2k},1,\overline{2,1}].
\end{align*}
Therefore, $\lambda_0^-(12_{2k-2}1\alpha_k^412_{2k-2}1):=C+D>A+B>m(\gamma_k^1)$.
\end{proof}

The case Rep3) essentially never occurs by Lemma \ref{te2-3.1} and the next lemma:
\begin{lemma}
If $a=2j<2k$, then $\lambda^-_0(12_{2j}1\alpha_k^412_{2k}12_3)\ge \lambda^-_0(12_{2k-2}1\alpha_k^412_{2k}12_3)>m(\gamma_k^1)$.
\end{lemma}
\begin{proof}
It is easy to see that $\lambda^-_0(12_{2j}1\alpha_k^412_{2k}12_3)\ge \lambda^-_0(12_{2k-2}1\alpha_k^412_{2k}12_3)$. In order to show that $\lambda^-_0(12_{2k-2}1\alpha_k^412_{2k}12_3)>m(\gamma_k^1)$, let $\underline{c}=2_{2k-2}12_{2k}12_{2k+1}12_{2k-1}12_{2k}12_3$ and $\underline{d}=12_{2k+1}12_{2k}12_{2k-1}12_{2k+1}12_{2k-2}$. We have
	$$\lambda^-_0(12_{2k-2}1\alpha_k^412_{2k}12_3):=A+B=[2;\underline{c},\overline{2,1}]+[0;\underline{d},1,\overline{1,2}]$$
and 
	$$m(\gamma^1_k)<[2;\underline{c},2_2,\overline{2,1}]+[0;\underline{d},2,\overline{2,1}]:=C+D.$$
Then,
	$$C-A=\dfrac{[2;\overline{2,1}]-[1;\overline{2,1}]}{q^2(\underline{c}2)([2;\overline{2,1}]+\beta(\underline{c2}))([1;\overline{2,1}]+\beta(\underline{c2}))}$$
	while
	$$B-D=\dfrac{[2;\overline{2,1}]-[1;\overline{1,2}]}{q^2(\underline{d})([2;\overline{2,1}]+\beta(\underline{d}))([1;\overline{1,2}]+\beta(\underline{d}))}.$$
	In particular,
	$$\dfrac{B-D}{C-A}=\dfrac{q^2(\underline{c}2)}{q^2(\underline{d})}\cdot X \cdot Y,$$
	where
	$$X=\dfrac{[2;\overline{2,1}]-[1;\overline{1,2}]}{[2;\overline{2,1}]-[1;\overline{2,1}]}>0.6339$$
	and 
	$$Y=\dfrac{([2;\overline{2,1}]+\beta(\underline{c}2))([1;\overline{2,1}]+\beta(\underline{c}2))}{([2;\overline{2,1}]+\beta(\underline{d}))([1;\overline{1,2}]+\beta(\underline{d}))}>0.82$$
	By Euler's rule, 
	$$q(\underline{c}2)>q(2_{2k-2}12_2)q(2_{2k-2}12_{2k+1}12_{2k-1}12_{2k}12_4) > 8q(2_{2k-3}1)q(2_{2k-2}12_{2k+1}12_{2k-1}12_{2k}12_4)$$ and 
	$$q(\underline{d})<2q(12_{2k-3})q(2_412_{2k}12_{2k-1}12_{2k+1}12_{2k-2}).$$ Thus, 
	$B-D>C-A$, that is, $A+B>C+D$. 
\end{proof}

The case Rep4) essentially never occurs by Lemma \ref{te2-3.2}, Lemma \ref{p1} i) and the next lemma:
\begin{lemma}
If $b=2m<2k$, then $\lambda^-_0(2212_{2k}1\alpha_k^412_{2m}1)\ge \lambda^-_0(2212_{2k}1\alpha_k^412_{2k-2}1)>m(\gamma_k^1)$.
\end{lemma}
\begin{proof}
By parity, it is easy to check that $\lambda^-_0(2212_{2k}1\alpha_k^412_{2m}1)\ge \lambda^-_0(2212_{2k}1\alpha_k^412_{2k-2}1)$. It remains to prove that $\lambda^-_0(2212_{2k}1\alpha_k^412_{2k-2}1)>m(\gamma_k^1)$.

Note that $\lambda^-_0(2212_{2k}1\alpha_k^412_{2k-2}1)=C+D$, where
\begin{align*}
C&:=[2;2_{2k-2},1,2_{2k},1,2_{2k+1},1,2_{2k-1},1,2_{2k-2},1,\overline{1,2}] {\rm \; and}\\
D&:=[0;1,2_{2k+1},1,2_{2k},1,2_{2k-1},1,2_{2k+1},1,2_{2k},1,2_2,\overline{1,2}].
\end{align*}
Moreover, by definition, we have $m(\gamma_k^1)<A+B$, where
\begin{align*}
A&:=[2;2_{2k-2},1,2_{2k},1,2_{2k+1},1,2_{2k-1},1,2_{2k},1,2_3,\overline{1,2}] {\rm \; and}\\
B&:=[0;1,2_{2k+1},1,2_{2k},1,2_{2k-1},1,2_{2k+1},1,2_{2k},1,2_3,\overline{1,2}].
\end{align*}

Hence, our work is reduced to prove that $C+D>A+B$. In order to prove this inequality, we observe that:
$$C-A=\dfrac{[2;2,1,2_3,\overline{1,2}]-[1;\overline{1,2}]}{\tilde{q}_{10k}^2([2;2,1,2_3,\overline{1,2}]+\tilde{\beta})([1;\overline{1,2}]+\tilde{\beta})},$$
and
$$B-D=\dfrac{[1;2,2,\overline{1,2}]-[1;2_3,\overline{1,2}]}{{q}_{10k+6}^2([1;2,2,\overline{1,2}]+{\beta})([1;2_3,\overline{1,2}]+{\beta})}$$
where $\tilde{q}_{10k}=q(2_{2k-2}12_{2k}12_{2k+1}12_{2k-1}12_{2k-2})$, ${q}_{10k+6}=q(12_{2k+1}12_{2k}12_{2k-1}12_{2k+1}12_{2k})$, $\tilde{\beta}=[0;2_{2k-2},1,2_{2k-1},1,2_{2k+1},1,2_{2k},1,2_{2k-2}]$ and ${\beta}=[0;2_{2k},1,2_{2k+1},1,2_{2k-1},1,2_{2k},1,2_{2k+1},1]$.
Thus,
$$\dfrac{C-A}{B-D}=\dfrac{[2;2,1,2_3,\overline{1,2}]-[1;\overline{1,2}]}{[1;2,2,\overline{1,2}]-[1;2_3,\overline{1,2}]}\cdot Y \cdot \dfrac{{q}^2_{10k+6}}{\tilde{q}^2_{10k}}>64.5\cdot Y \cdot \dfrac{{q}^2_{10k+6}}{\tilde{q}^2_{10k}},$$
where
$$Y=\dfrac{([1;2,2,\overline{1,2}]+{\beta})([1;2_3,\overline{1,2}]+{\beta})}{([2;2,1,2_3,\overline{1,2}]+\tilde{\beta})([1;\overline{1,2}]+\tilde{\beta})}>\dfrac{([1;2_2,\overline{1,2}]+[0;2_4])([1;2_3,\overline{1,2}]+[0;2_4])}{([2;2,1,2_3,\overline{1,2}]+[0;2_3])([1;\overline{1,2}]+[0;2_3])}>0.56.$$
Let $\Gamma=2_{2k-2}12_{2k}12_{2k+1}1$ and $\Sigma=2_{2k-1}12_{2k-2}$. By Euler's rule, we have:
\begin{align*}
{q}_{10k+6}&>q(\Gamma^t2)q(12_{2k+1}12_{2k})>2q(\Gamma^t)q(12_{2k+1}12_{2k-2})q(2_2)=10q(\Gamma^t)q(2_{2k-2}12_{2k+1}1)>\\
&>10q(\Gamma^t)q(2_{2k-2}12_{2k-1})q(2_21)=70q(\Gamma^t)q(\Sigma^t),
\end{align*} 
and 
\begin{align*}
\tilde{q}_{10k}<2q(\Gamma)q(\Sigma).
\end{align*} 

Thus, 
$$\dfrac{C-A}{B-D}> 64.50 \cdot 0.56 \cdot \left( 35 \right)^2>1.$$
\end{proof}

An immediate consequence of the previous three lemmas is the fact that essentially only the case Rep1) occurs:
\begin{corollary} \label{crl2}
There is an explicit constant $\lambda_k^{(9)}>m(\gamma_k^1)$ such that the neighbourhood of the string $2_21\alpha_k^412_4$ in any $(k,\lambda_k^{(9)})$-admissible word $\theta$ has the form $$\theta=...2212_{2k}\alpha_k^412_{2k}12_4=...2212_{2k}12_{2k+1}12_{2k-1}12_{2k}12_{2k+1}12_{2k-1}12_{2k}12_{2k+1}12_{2k-1}12_{2k}12_4....$$
\end{corollary}

\subsection{Extension from $2212_{2k}1\alpha_k^412_{2k}12_{4}$ to $2212_{2k-1}12_{2k}1\alpha_k^412_{2k}12_{4}$}

Let $\theta=...2212_{2k}1\alpha_k^412_{2k}12_{4}...$. It extends as $\theta=...2_{a}12_{2k}1\alpha_k^412_{2k}12_{4}...$. By Lemma \ref{t2-3} ii),  we have that  $a\le 2k-1$. Using Lemma \ref{bpodd}, we have that if $a$ is odd, then $a=2k-1$. Moreover, by Lemma \ref{e3.c.e}, we can not have $a=2j<2k-1$.

\begin{corollary} \label{crl3}
There exists an explicit constant $\lambda_k^{(10)}>m(\gamma_k^1)$ such that the neighbourhood of the string $2212_{2k}1\alpha_k^412_{2k}12_{4}$ in any $(k,\lambda_k^{(10)})$-admissible word $\theta$ has the form $\theta=...2212_{2k-1}12_{2k}1\alpha_k^412_{2k}12_4=$
$$=...2212_{2k-1}12_{2k}12_{2k+1}12_{2k-1}12_{2k}12_{2k+1}12_{2k-1}12_{2k}12_{2k+1}12_{2k-1}12_{2k}12_4....$$
\end{corollary}

\subsection{Extension from $2212_{2k-1}12_{2k}1\alpha_k^412_{2k}12_{4}$ to $2212_{2k+1}12_{2k-1}12_{2k}1\alpha_k^412_{2k}12_{4}$}

Let $\theta =...2212_{2k-1}12_{2k}1\alpha_k^412_{2k}12_{4}...$. It extends as $\theta=...2_{a}12_{2k-1}12_{2k}1\alpha_k^412_{2k}12_{4}...$. By Lemma \ref{srl} ii),  we have that  $a\le 2k+1$. By Lemma \ref{e4.c.e}, we can not have $a=2m<2k+1$.  Using Lemma \ref{bpodd}, we have that if $a$ is odd, then $a\ge 2k-1$. Finally, by Lemma \ref{ooe} i), we can not have $a=2k-1$. Thus, we have the following corollary:

\begin{corollary} \label{crl4}
Consider the parameter
$$\lambda_k^{11}:=\min\{ \lambda^-_0(11\alpha_k^412_4),  \lambda^-_0(\Delta_{2k-2}), \lambda^-_0(2_{2k-2}12^*2_{2k-4}1), \lambda^-_0(112_{2k-1}12^*2_{2k-2}122)\}.$$
Then, $\lambda_k^{11}>m(\gamma_k^1)$ and the neighbourhood of the string $2212_{2k-1}12_{2k}1\alpha_k^412_{2k}12_{4}$ in any $(k,\lambda_k^{(10)})$-admissible word $\theta$ has the form $\theta=...2212_{2k+1}12_{2k-1}12_{2k}1\alpha_k^412_{2k}12_{4}=$ $$=...2212_{2k+1}12_{2k-1}12_{2k}12_{2k+1}12_{2k-1}12_{2k}12_{2k+1}12^*2_{2k-2}12_{2k}12_{2k+1}12_{2k-1}12_{2k}12_4....$$
\end{corollary}
The discussion on this section can be summarised into the following lemma establishing the self-replication property of $\gamma_k^1$ for all $k\geq 4$:

\begin{lemma}[Replication Lemma]\label{l.replicamento} For each natural number $k\ge 4$, there exists an explicit constant $\nu_k^{(1)}>m(\gamma_k^1)$ such that any $(k,\nu_k^{(1)})$-admissible word $\theta$ containing $\alpha_k^4:=2_{2k+1}12_{2k-1}12_{2k}12_{2k+1}12_{2k-1}12_{2k}12_{2k+1}12_{2k-1}$ must extend as 
$$\theta=...2212_{2k+1}12_{2k-1}12_{2k}12_{2k+1}12_{2k-1}12_{2k}12_{2k+1}12_{2k-1}12_{2k}12_{2k+1}12_{2k-1}12_{2k}12_4...$$
and the neighbourhood of the position $-(6k+3)$ is $$...2_212_{2k+1}12_{2k-1}12_{2k}12_{2k+1}12_{2k-1}12_{2k}12_{2k+1}12_{2k-1}....$$ In particular, any  $(k,\nu_k^{(1)})$-admissible word $\theta$ containing $\alpha_k^4$ has the form
$$\overline{2_{2k-1}12_{2k}12_{2k+1}1}2^*2_{2k-2}12_{2k}12_{2k+1}12_{2k-1}12_{2k}12_4$$
\end{lemma}
\begin{proof}
This result for $\nu_k^{(1)}:=\min \{\lambda_k^{(i)}: i=8,...,11\}$ is a consequence of Corollaries \ref{crl1}, \ref{crl2}, \ref{crl3} and \ref{crl4}.
\end{proof}

\section{End of the proof of Theorem \ref{t.arabismos}}\label{s.end}

By Lemma \ref{l.2-1} and Proposition \ref{p.3-1}, we have that the Markov values $m(\theta(\underline{\omega}_k))=\lambda_0(\theta(\underline{\omega}_k))$ and $m(\gamma_k^1)=\lambda_0(\gamma_k^1)$ satisfy $m(\theta(\underline{\omega}_k)) < m(\gamma_k^1) < m(\theta(\underline{\omega}_{k-1}))$ for all $k\geq 3$ and $\lim\limits_{k\to\infty}m(\theta(\underline{\omega}_{k}))=1+3/\sqrt{2}$. 

Moreover, we affirm that $m(\gamma_k^1)\notin L$ for all $k\geq 4$. Indeed, it follows from Theorems \ref{t.local-uniqueness},  \ref{t.extension} and Lemma \ref{l.replicamento} that if $\lambda_k:=\min\{\lambda_k^{(1)},\mu_k^{(1)}, \nu_k^{(1)}\}$, then any element $\ell\in L$ with $m(\theta(\underline{\omega}_{k}))<\ell<\lambda_k$ would necessarily have the form $\ell=m(\overline{2_{2k-1}12_{2k}12_{2k+1}1})=m(\theta(\underline{\omega}_{k}))$, a contradiction. This completes the proof of the desired theorem. 

\begin{remark}\label{r.isolated-L-proof} For each $k\geq 4$, our arguments above were based on the construction of a \emph{finite} set of $k$-prohibited and $k$-avoided strings. In particular, we proved that there is also an explicit constant $\rho_k<m(\theta(\underline{\omega}_k))$ such that the statements of Theorems \ref{t.local-uniqueness},  \ref{t.extension} and Lemma \ref{l.replicamento} are valid for any word $\theta$ with $\rho_k<m(\theta)=\lambda_0(\theta)<\lambda_k$. Thus, an element $\ell\in L$ with $\rho_k<\ell<\lambda_k$ has the form $\ell=m(\overline{2_{2k-1}12_{2k}12_{2k+1}1})=m(\theta(\underline{\omega}_{k}))$ and, \emph{a fortiori}, $m(\theta(\underline{\omega}_{k}))$ is an isolated point of $L$. 
\end{remark}

\bibliographystyle{amsplain}

 \end{document}